\documentclass[11pt,a4paper,headsepline]{article}
\usepackage[ngerman,british]{babel}
\usepackage[utf8]{inputenc}
\usepackage{amsmath}
\usepackage{amsthm}
\usepackage{amsfonts}
\usepackage{amssymb}
\usepackage{dsfont}
\usepackage[T1]{fontenc}
\usepackage{makeidx}
\usepackage{units}
\usepackage[pdftex]{graphicx} 
\usepackage[dvipsnames]{xcolor}
\usepackage{graphicx}
\usepackage{floatflt} 
\usepackage{verbatim}
\usepackage{wrapfig} 
\usepackage{caption}
\usepackage{latexsym}
\usepackage{setspace}
\usepackage{hyperref}
\usepackage{enumitem}

\allowdisplaybreaks[4]

\newtheorem{theorem}{Theorem}[section]
\newtheorem{defin}[theorem]{Definition}
\newtheorem{Prop}[theorem]{Proposition}
\newtheorem{lem}[theorem]{Lemma}

\theoremstyle{definition}
\newtheorem{ex}[theorem]{Example}
\newtheorem{remark}[theorem]{Remark}
\newtheorem{remarks}[theorem]{Remarks}

\theoremstyle{plain}

\newcommand{\NN}{\mathbb{N}}

\newcommand{\RR}{\mathbb{R}}
\newcommand{\PP}{\mathbb{P}}

\newcommand{\HH}{\mathbb{H}}
\newcommand{\FF}{\mathbb{F}}
\newcommand{\GG}{\mathbb{G}}
\newcommand{\TT}{\mathbb{T}}
\newcommand{\EE}{\mathbb{E}}

\newcommand{\cS}{\mathcal{S}}

\newcommand{\cF}{\mathcal{F}}
\newcommand{\cG}{\mathcal{G}}

\newcommand{\cE}{\mathcal{E}}

\newcommand{\cL}{\mathcal{L}}
\newcommand{\FE}{\mathfrak{E}}

\newcommand{\taure}{\tau^{re}}
\newcommand{\tauex}{\tau^{ex}}

\newcommand{\fatQ}{\mathbf{Q}}

\newcommand{\frace}{\mathfrak{E}}
\newcommand{\fracef}{\mathfrak{F}}

\newcommand{\iv}{\overset{d}{=}}

\DeclareMathOperator*{\asure}{\;a.s.}
\newcommand{\inV}{\overset{d}{\to}}
\DeclareMathOperator*{\inPpi}{\overset{\mathbb{P}_\pi}{\to}}

\DeclareMathOperator{\Exp}{Exp}

\begin{document}
	\title{Markov-modulated generalized Ornstein-Uhlenbeck processes\\ and an application in risk theory}
	\author{Anita Behme\thanks{Technische Universit\"at
			Dresden, Institut f\"ur Mathematische Stochastik, Zellescher Weg 12-14, 01069 Dresden, Germany, \texttt{anita.behme@tu-dresden.de} and \texttt{apostolos.sideris@tu-dresden.de}, phone: +49-351-463-32425, fax:  +49-351-463-37251.}\; and Apostolos Sideris$^\ast$}
	\date{\today}
	\maketitle
	
	\vspace{-1cm}
	\begin{abstract}
	 We derive the Markov-modulated generalized Ornstein-Uhlenbeck process by embedding a Markov-modulated random recurrence equation in continuous time. The obtained process turns out to be the unique solution of a certain stochastic differential equation driven by a bivariate Markov-additive process. We present this stochastic differential equation as well as its solution explicitely in terms of the driving Markov-additive process. Moreover, we give necessary and sufficient conditions for strict stationarity of the Markov-modulated generalized Ornstein-Uhlenbeck process, and prove that its stationary distribution is given by the distribution of a specific exponential functional of Markov-additive processes. Finally we propose an application of the Markov-modulated generalized Ornstein-Uhlenbeck process as Markov-modulated risk model with stochastic investment. This generalizes Paulsen's risk process to a Markov-switching environment. We derive a formula in this risk model that expresses the ruin probability in terms of the distribution of an exponential functional of a Markov-additive process.
	\end{abstract}
	
	2020 {\sl Mathematics subject classification.} 60H10, 60J25, 60G51, (primary), 60J57, 91G05, 60K37 (secondary)\\
	
	{\sl Keywords:} exponential functional; generalized Ornstein-Uhlenbeck process; L\'evy process; Markov additive process; Markov-modulated random recurrence equation; Markov-switching model; risk theory; ruin probability; stationary process
	
	\section{Introduction}\label{S0}
	\setcounter{equation}{0}
	
	Given a bivariate Lévy process $(\xi,\eta)=(\xi_t,\eta_t)_{t\geq 0}$,  the \emph{generalized Ornstein-Uhlenbeck (GOU) process}  $(V_t)_{t\geq 0}$ driven by $(\xi,\eta)$ is defined as 
	\begin{align}\label{GOULevyexplizit}
	V_t=e^{-\xi_t}\left(V_0+\int_{(0,t]} e^{\xi_{s-}}d\eta_s\right) ,\quad t\geq 0.
	\end{align}
	This class of processes has been introduced in 1988 by de Haan and Karandikar \cite{DEHAAN+KARANDIKAR_EmbeddingSDEcontinuoustimeprocess} as continuous-time analogue to discrete-time solutions of certain random recurrence equations. Moreover, it has been shown in  \cite{DEHAAN+KARANDIKAR_EmbeddingSDEcontinuoustimeprocess} that $(V_t)_{t\geq 0}$ is the unique solution of the SDE 
	\begin{align}\label{GOUSDELevy}
	dV_t=V_{t-}dU_t+dL_t,\quad t\geq 0,
	\end{align}
	for another bivariate Lévy process $(U_t,L_t)_{t\geq 0}$, which is in direct relation to $(\xi_t,\eta_t)_{t\geq 0}$.\\
	 Special cases of GOU processes date back to the early 20th century, cf. \cite{EINSTEIN_History1905} and \cite{ORNSTEIN+UHLENBECK_kHistory1930}. Nowadays, GOU processes have become a permanent tool in stochastic modelling, appearing in numerous applications such as finance and insurance, see \cite{BNS01}, \cite{KLUEPPELBERG+LINDNER+MALLER_COGARCHstationarityandSecondorderBehaviour2004}, or \cite{PAULSEN_RisksStochasticEnvironment1992} for some examples. 
		 The properties of GOU processes have been studied thoroughly in many papers, see for instance \cite{BEHME+LINDNER+MALLER_StationarySolutionsSDEGOU},  \cite{Kevei}, \cite{LINDNER+MALLER_LevyintegralsStationarityGOUP}, or \cite{MALLER+MUELLER+SZIMAYER_OUPExtensions}. 
	
Also dating back to the 80's, see e.g. \cite{Janssen}, \cite{hamilton} or \cite{Reinhard}, Markov-switching models have become a popular tool in finance and other areas;  a collection of possible applications being presented in \cite{Hamiltonbook}. It is thus natural to aim for GOU processes with a Markov-switching behaviour. 
Several studies in this direction exist. In \cite{HUANG+SPREJ_MarkovmodulatedOUP2014}, \cite{LINDSKOG+MAJUMDER_ExactLongtimebehaviourMAP2019}, and \cite{ZHANG+WANG_StationarydistributionOUP2stateMarkovswitching2017} the authors consider so-called Markov-modulated Ornstein-Uhlenbeck (MMOU) processes, that is Markov-modulated versions of the classical Ornstein-Uhlenbeck process, i.e. of \eqref{GOULevyexplizit} for a deterministic process $\xi_t=\lambda t$, $t\geq 0$, and a Brownian motion with drift $\eta_t=\gamma_t + \sigma B_t$, $t\geq 0$. In \cite{ZHAO_MarkovmodulatedOUPwithapplicationsinAlberta2012} a Markov-modulated version of the Lévy-driven Ornstein-Uhlenbeck process (i.e. \eqref{GOULevyexplizit} with $\xi_t=\lambda t$, $t\geq 0$ and general Lévy process $(\eta_t)_{t\geq 0}$) is used to model electricity spot prices. 
However, despite the fact that \cite{DEHAAN+KARANDIKAR_EmbeddingSDEcontinuoustimeprocess} and \cite{hamilton} have been published about 30 years ago and in the meantime GOU processes and Markov-modulated models have become prominent tools for numerous applications, up to now there exists no thorough theoretical description of a GOU process with Markov-modulated behaviour in the literature. 

 It is the first aim of this paper to close this gap, and thus in Section \ref{S1}, we will define a Markov-modulated version of the GOU process following the approach of \cite{DEHAAN+KARANDIKAR_EmbeddingSDEcontinuoustimeprocess}. That is, we derive the  \emph{Markov-modulated generalized Ornstein-Uhlenbeck (MMGOU)} process as continuous-time analogue to discrete-time solutions of Markov-modulated random recurrence equations. Hereby we will allow for a rather general Markov-modulation in the sense that the background driving Markov chain may have a countably infinite state space, and that jumps in the background driving Markov chain may induce additional jumps in $(\xi,\eta)$.
	
	Having defined the MMGOU process it is our second purpose to determine necessary and sufficient conditions for strict stationarity of the process and to describe its stationary distribution. Apart from the classical case of the Lévy-driven GOU process \eqref{GOULevyexplizit} studied in \cite{LINDNER+MALLER_LevyintegralsStationarityGOUP} and \cite{BEHME+LINDNER+MALLER_StationarySolutionsSDEGOU}, to our knowledge this topic is only covered in the literature  for the special case of a Markov-modulated Ornstein-Uhlenbeck process, see e.g. \cite{LINDSKOG+MAJUMDER_ExactLongtimebehaviourMAP2019}, or  \cite{ZHANG+WANG_StationarydistributionOUP2stateMarkovswitching2017}. We provide necessary and sufficient conditions for the existence of a strictly stationary MMGOU process in Section \ref{S3} of this paper and prove that - as in the Lévy case - given it exists, its stationary distribution can be described by the distribution of a certain exponential functional. However, in the Markov-modulated situation the exponential functional is driven by time-reverted processes and does not necessarily have to converge in an almost sure sense to ensure stationarity. 
	
	We end this paper with a short study of a Markov-modulated risk model that fits nicely into the framework of MMGOU processes and exponential functionals of MAPs. Our model combines a Markov-modulated version of the classical Cram\'er-Lundberg risk process with an investment possibility that is also modulated by the background driving Markov chain. It can therefore be seen as a generalization of Paulsen's risk process \cite{PAULSEN_RisksStochasticEnvironment1992} to the Markov-switching situation.\\ Note that first Markov-modulated risk models have already been studied in \cite{Janssen} and \cite{Reinhard} in the 80's. The first Markov-modulated risk process where risk reserves can be invested into a stock index following a geometric Brownian motion has been introduced by \cite{Baeuerle}. In \cite{Ramsden} a generalization of this model has been considered, in which the investment returns still follow a geometric Brownian motion, but this is also influenced by the external Markov chain. The developed theory of the MMGOU process now allows to go one more step further, in that investment returns may be generated by a process with jumps, instead of just a Brownian motion. Moreover, other than in the mentioned studies, our model allows for joint jumps of the modulating Markov chain, the surplus generating process and the investments generating process that may be interpreted as market shocks at times of regime switches. In Theorem \ref{thm-ruinMMpaulsen} we present a formula for the ruin probability in model. This formula is based on the distribution of a stochastic exponential of a bivariate MAP associated to the model; in other words, the ruin probability is governed by the stationary distribution of an associated MMGOU process.

	\section{The definition of a Markov-modulated GOU process}\label{S1}
	\setcounter{equation}{0}
	
	Extending an earlier result from Wolfe \cite{WOLFE_ContinuousanalogueOUSDE},  De Haan and Karandikar  \cite{DEHAAN+KARANDIKAR_EmbeddingSDEcontinuoustimeprocess} introduced the generalized Ornstein-Uhlenbeck process \eqref{GOULevyexplizit} as continuous-time analogue of a solution to a random recurrence equation with i.i.d. coefficients. In order to derive the Markov-modulated GOU process in a similar way as a continuous-time analogon to solutions of Markov-modulated random recurrence equations in Section \ref{S2} below, we first recall some basic facts on Markov additive processes in Section \ref{S1a} and derive some preliminary results concerning the stochastic exponential of Markov additive processes in Section \ref{S1b}. 
	
	\subsection{Some preliminaries on Markov additive processes}\label{S1a}
	
	Throughout the paper, let $(\Omega,\mathcal{F},\mathbb{F},\PP)$ be a filtered probability space and $(S,\mathcal{S})$ a measurable space, where we assume $S$ to be at most countable. Let $(X,J)=(X_t,J_t)_{t\geq 0}$ be a Markov process on $\RR^d\times S$, $d\geq 1$. The filtration $\mathbb{F}=\left(\mathcal{F}_t\right)_{t\geq 0}$ shall always satisfy the usual conditions, and if not stated otherwise we choose it  to be the smallest filtration satisfying the usual conditions for which $(X,J)$ is adapted. For any $j\in S$ we write $\PP_j(\cdot):=\PP(\cdot|J_0=j)$ and $\EE_j[\cdot]$ for the expectation with respect to $\PP_j$. Convergence in $\PP_j$-probability will be denoted as $\overset{\PP_j}\longrightarrow$. Equality and convergence in distribution will be denoted as $\overset{d}=$ and $\overset{d}\longrightarrow$, respectively.
		
	We will use the standard definition of a Markov additive process, cf. \cite{ASMUSSEN_AppliedProbandQueues}, that we present here in a multivariate setting as it has been already considered e.g. in \cite{CINLAR_MAP11972, CINLAR_MAP21972}.

	\begin{defin} \label{DefMAP} A $d+1$-dimensional continuous time Markov process $(X_t,J_t)_{t\geq 0}$ on $\RR^d\times S$ is called a \emph{($d$-dimensional) Markov additive process with respect to $\FF$ ($\FF$-MAP)}, if for all $s,t\geq 0$ and for all bounded and measurable functions $f:\RR^d\to\RR$, $g:S\to \RR$
	\begin{align}\label{MAPdefinition}
	\EE_{J_0}\left[f(X_{s+t}-X_s)g(J_{s+t})|\mathcal{F}_s\right]=\EE_{J_s}\left[f(X_t)g(J_t)\right].
	\end{align}
	\end{defin}
	
	Given a MAP $(X,J)$ the marginal process $J$ is often called \emph{Markovian component} or \emph{Markov driving chain/process}, sometimes the name \emph{modulator} is use. The process $X$ is typically called the \emph{additive component}, and in practice this is the process of interest. 
	
	In our derivations we shall also rely on the following extension of \eqref{DefMAP}. 
	
	\begin{lem}\label{lem-MAPproperty} The condition \eqref{MAPdefinition} in Definition \ref{DefMAP} implies (and is thus equivalent to)
		\begin{equation} \label{MAPpropextend}
		\PP_{J_0}(C_s, J_{s+t}\in D|\cF_s) = \PP_{J_s}(C_0, J_{t}\in D)		
		\end{equation}
		for all $C_s\in \sigma(X_{s+u}-X_s, 0\leq u\leq t)$ and corresponding events $C_0\in \sigma(X_u, 0\leq u\leq t)$, and for all $D\in \cS$.
	\end{lem}
	\begin{proof}
		Note first that it follows by induction from \eqref{MAPdefinition} for all $n\in \NN$, $0\leq s, 0=t_0\leq t_1\leq t_2\leq \ldots \leq t_n$ and bounded, measurable functions $f_1,\ldots, f_n, g$, that
		\begin{align*}
		\EE_{J_0}\left[\left(\prod_{k=1}^{n} f_k\left(X_{s+t_k}-X_{s+t_{k-1}} \right) \right) g(J_{s+t_n})|\mathcal{F}_s\right]=\EE_{J_s}\left[\left(\prod_{k=1}^{n} f_k\left(X_{t_k}-X_{t_{k-1}}\right)\right)g(J_{t_n})\right].
		\end{align*}
		This may be rewritten as
		\begin{align*}\PP_{J_0}\left( X_{s+t_1}- X_s \in B_1, \ldots,\right. & \left.  X_{s+t_n}-X_{s+t_{n-1}} \in B_n, J_{s+t_n}\in D |\cF_s  \right) \\
		&= \PP_{J_s} \left( X_{t_1} \in B_1, \ldots, X_{t_n}- X_{t_{n-1}} \in B_n, J_{t_n}\in D \right) \end{align*}
		for any Borel sets $B_1,\ldots, B_n \subset \RR\setminus\{0\}$ and $D\in \cS$. As the occuring events 
		$\{X_{s+t_k}- X_{s+t_{k-1}} \in B_k\}$ generate the sigma-fields
		$\sigma(X_{s+u}- X_s, 0\leq u\leq t)$ this implies by uniqueness of measures that
		$$\PP_{J_0}(C_s, J_{s+t_n}\in C|\cF_s) = \PP_{J_s}(C_0, J_{t_n}\in D)$$
		for all $C_s\in \sigma(X_{s+u}-X_s, 0\leq u\leq t)$ and corresponding events $C_0\in \sigma(X_u, 0\leq u\leq t)$ as claimed.
	\end{proof}

	From the definition of the MAP $(X,J)$ one can show (see \cite{ASMUSSEN_AppliedProbandQueues} for the case of $S$ being finite or \cite{CINLAR_MAP11972} for the original treatment), that since $S$ is at most countable, there exists a sequence of independent $d$-dimensional Lévy processes $\lbrace X^j,j\in S\rbrace$ with characteristic triplets $\lbrace(\gamma_{X^j},\Sigma^2_{X^j},\nu_{X^j}):j\in S\rbrace$ such that, whenever $J_t=j$ on some time interval $(t_1,t_2)$, the additive component $(X_t)_{t_1<t<t_2}$ behaves in law as $X^j$. To simplify notation we will also write the triplets as $(\gamma_{X}(j),\Sigma^2_{X}(j),\nu_{X}(j)):=(\gamma_{X^j},\Sigma^2_{X^j},\nu_{X^j})$.  \\
	Moreover, whenever the driving chain $J$ jumps, say at time $T_n$, from state $i$ to state $j$, it induces an additional jump $\Phi^{ij}_{X,n}$ for $X$, whose distribution $F_{X,n,T_n}^{ij}$ depends only on $(i,j)$ and neither on the jump time nor on the jump number, i.e. $F^{ij}_{X,n,T_n}= F_X^{ij}$, for some distribution $F_X^{ij}$, and it is independent of any other occurring random elements. This implies that we can always find a c\`adl\`ag modification of $(X,J)$ and we will therefore assume any MAP $(X,J)$ to be c\`adl\`ag from now on. Moreover it leads to the following path decomposition of the additive component	
	\begin{align}\label{MAPpathdescription}
	X_t=X_0+X_t^{(1)}+X_t^{(2)}=X_0+ \int_{(0,t]}   dX_{s-}^{J_{s-}}+\sum_{n\geq 1}\sum_{i,j\in S}\Phi_{X,n}^{ij} \mathds{1}_{\lbrace J_{T_{n-1}}=i,J_{T_{n}}=j,T_n\leq t\rbrace},
	\end{align}
	where $(T_n)_{n\geq 0}$ denotes the sequence of jump times of $J$ and $T_0=0$. Conversely, it can be easily checked that every process $(X,J)$, such that $(J_t)_{t\geq 0}$ is a continuous-time Markov chain with state space $S$ and $X$ has a representation as in \eqref{MAPpathdescription}, is a MAP. Furthermore, note that the process $(X^{(1)}_t)_{t\geq 0}$ in \eqref{MAPpathdescription} is a semimartingale whose characteristics are functionals of $J$, cf. \cite{grigelionis}. As $(X_t^{(2)})_{t\geq 0}$ clearly is a process of finite total variation and thus a semimartingale as well (cf. \cite[Thm. II.7]{PROTTER_StochIntandSDE}), we may and will use the additive component  $X$ of a MAP $(X,J)$ as stochastic integrator.\\
	We will always assume that $X_0=0$. We refer to $X^{(1)}$ as the \emph{pure switching part} of $X$.  In the case that $\Phi_{X,n}^{ij}\equiv 0$ a.s. for all $n\geq 0$ and all $i,j\in S$ we call $X$ a \emph{pure switching MAP}. 
	
	In this paper, we will mostly work with MAPs that have a two-dimensional additive component, i.e. we consider $(X,J)=((\zeta,\chi),J)=((\zeta_t,\chi_t),J_t)_{t\geq 0}\in \RR^2\times S$. For such bivariate MAPs the Lévy processes $X^j=(\zeta^j, \nu^j)$ and the additional jumps $\Phi_{X,n}^{ij}$ are two-dimensional. In particular in this case the triplets $(\gamma_{X^j},\Sigma^2_{X^j},\nu_{X^j})$ consist of a non-random vector $(\gamma_{\zeta^j},\gamma_{\chi^j})=(\gamma_{\zeta}(j),\gamma_{\chi}(j))\in\mathbb{R}^2$, a symmetric, non-negative definite $2\times 2$ matrix 
	\begin{align*}
	\Sigma^2_{X^j}=\Sigma^2_X(j)=:\begin{pmatrix}
	\sigma^2_{\zeta^j,\zeta^j}&\sigma_{\zeta^j,\chi^j}\\
	\sigma_{\zeta^j,\chi^j}&\sigma^2_{\chi^j,\chi^j}
	\end{pmatrix}=\begin{pmatrix}
	\sigma^2_{\zeta}(j)&\sigma_{\zeta,\chi}(j)\\
	\sigma_{\zeta,\chi}(j)&\sigma^2_{\chi}(j)
	\end{pmatrix}
	\end{align*}
	and a Lévy measure $\nu_{X^j}$ on $\mathbb{R}^2\setminus\lbrace(0,0)\rbrace$. Note that in this article we will switch between the notations of row vectors $(\zeta,\chi)$ and the corresponding column vectors $\left(\begin{smallmatrix} \zeta \\ \chi \end{smallmatrix}\right)$ depending on readability only - without marking transpositions. 
	
	For a more detailed description of Lévy processes and their characteristics we refer to \cite{SATO_LPinfinitelydivisibledistributions}.	For more details on MAPs we refer to Section \ref{S3a} below, the seminal works of \c{C}inlar \cite{CINLAR_MAP11972},\cite{CINLAR_MAP21972}, the modern textbook treatment in \cite{ASMUSSEN_AppliedProbandQueues}, and the comprehensive Appendix of \cite{DEREICH+DOERING+KYPRIANOU_RealselfsimilarprocessesStartedfromOrigin2017}.
	
	\subsection{Stochastic exponentials and Markov multiplicative processes}\label{S1b}
	
	Apart of Markov additive processes our derivation of the Markov-modulated GOU process also relies on a class of Markov processes that we will call \emph{Markov multiplicative processes} (see Definition \ref{DefMMP} below). In this subsection, we will introduce these and study their deep relation to the class of Markov additive processes.
	
	Recall that for any real-valued semimartingale $U=(U_t)_{t\geq 0}$ the \emph{(Doléans-Dade) stochastic exponential} of $U$, denoted by $\cE(U)=(\cE(U)_t)_{t\geq 0}$, is defined as the unique solution of the SDE
	\begin{align}\label{StochExpSDE}
	dZ_t=Z_{t-}dU_t,\;t >0,\quad Z_0=1.
	\end{align}
	It has the explicit representation (see e.g. \cite[Thm. II.37]{PROTTER_StochIntandSDE})
	\begin{align} \label{exponentialexplicit}
	\mathcal{E}(U)_t=\exp\left(U_t-\frac{1}{2}[ U^c,U^c]_t\right)\prod_{0<s\leq t}\left(1+\Delta U_s\right)e^{-\Delta U_s}, \quad t\geq 0,
	\end{align}
	with $U^c$ denoting the continuous local martingale part of $U$, $[\cdot,\cdot]$ the quadratic variation process, and $\Delta U_t:=U_t-U_{t-}$ denoting the jump of $U$ at time $t>0$. It is well known, and directly observable from the above formula, that the stochastic exponential remains strictly positive if and only if $U$ has no jumps of size less or equal to $-1$, while $U$ having no jumps of size $-1$ is equivalent to $\cE(U)_t \neq 0$ for all $t\geq 0$. \\
	Furthermore we recall the \emph{stochastic logarithm} $\mathcal{L}og(Z)$ of an $\RR\backslash\{0\}$-valued semimartingale $Z=(Z_t)_{t\geq 0}$ that is defined via  $\cL og(Z)_t=\int_{(0,t]}\frac{1}{Z_{s-}}dZ_s$, $t\geq 0$. It is clear from this definition and \eqref{StochExpSDE} that if $U$ is a semimartingale with no jumps of size $-1$, then $\cL og(\cE(U))=U$ a.s. 
	
	Via the exponential functional we may establish a close connection between one-dimensional MAPs and a class of stochastic processes that we will call \textit{Markov multiplicative processes}. In analogy to Definition \ref{DefMAP} we define them as follows. 
	
	\begin{defin}\label{DefMMP}
		A bivariate continuous time Markov process $(Z,J)=(Z_t,J_t)_{t\geq 0}$ on $\RR\backslash\{0\}\times S$ is a \emph{Markov multiplicative process with respect to the filtration $\FF$ ($\FF$-MMP)}, if for all $s,t\geq 0$ and for all bounded and measurable functions $f,g$
		\begin{align}\label{DefMMPproperty}
		\EE_{J_0}\left[f\left(\frac{Z_{s+t}}{Z_s}\right)g(J_{s+t})|\mathcal{F}_s\right]=\EE_{J_s}\left[f(Z_t)g(J_t)\right].
		\end{align}
		If $(Z,J)$ is an MMP, we refer to $Z$ as its \emph{multiplicative component}, while $J$ is its \emph{Markovian component}.
		\end{defin}
	
	Note that although a general definition of multivariate MMPs could have been given similarly, we restrict ourselves to the one-dimensional case, as we will not need the multivariate version in this paper. This allows us to avoid unnecessary technicalities that would arise when considering non-commutative matrix multiplications. 
	
	In analogy to Lemma \ref{lem-MAPproperty} one can prove the following.
	
	\begin{lem}\label{lem-MMPproperty} The condition \eqref{DefMMPproperty} in Definition \ref{DefMMP} implies (and is thus equivalent to)
		\begin{equation*}
		\PP_{J_0}(C_s, J_{s+t}\in D|\cF_s) = \PP_{J_s}(C_0, J_{t}\in D)		
		\end{equation*}
		for all events $C_s\in \sigma(Z_{s+u}/Z_s, 0\leq u\leq t)$ and corresponding events $C_0\in \sigma(Z_u, 0\leq u\leq t)$, and for all $D\in \cS$.
		\end{lem}

	The following theorem establishes a one-to-one correspondence between MMPs and MAPs whose additive components have no jumps of size $-1$. In the case of Lévy processes (i.e. for $|S|=1$) such a relation between additive and multiplicative processes  is known and goes back to results of Skorokhod \cite{skorokhod}.
	
	\begin{theorem}\label{1-1MAP-MMP}
A stochastic process $(Z,J)$ on $\RR\backslash\{0\}\times S$ is a Markov multiplicative process with respect to $\FF$ if and only if there exists a Markov additive process $(U,J)$ with respect to $\FF$ on $\RR\times S$ such that $\Delta U\neq -1$ and $Z=\mathcal{E}(U)$.
	\end{theorem}

Note that this deep relation between MAPs and MMPs directly implies, that the multiplicative component of every MMP admits a càdlàg modification and we will therefore assume every multiplicative component of an MMP to be càdlàg from now on. Moreover, as every multiplicative component of an MMP can be identified with a stochastic exponential, we observe that it is a semimartingale.

\begin{proof}[Proof of Theorem \ref{1-1MAP-MMP}]
	Assume first that $(U,J)$ is a MAP with $\Delta U\neq -1$ and consider the auxiliary process $((Y,K),J)$ defined by
	\begin{align*}
	Y_t&:=-U_t+\frac{1}{2}\int_{(0,t]} \sigma_U^2(J_s)ds+\sum_{0<s\leq t} \left(\Delta U_s-\log|1+\Delta U_s|\right),\quad  t\geq 0,\\
	K_t&:=\sum_{0<s\leq t} \mathds{1}_{\lbrace \Delta U_s<-1\rbrace} , \quad t\geq 0.
	\end{align*}
	Observe that the path decomposition \eqref{MAPpathdescription} of $U$ leads to a path decomposition of $(Y,K)$ that is again of the form \eqref{MAPpathdescription}. Thus $((Y,K),J)$ is a bivariate MAP with respect to $\FF$. Furthermore,  a  direct computation using  \eqref{exponentialexplicit} shows that $Z_t=\cE(U)_t=(-1)^{K_t}e^{-Y_t}$  a.s. for all $t\geq 0$ since $[U^c,U^c]_t= \int_{(0,t]} \sigma_U^2(J_s) ds$, $t\geq 0$, which again is a direct consequence of \eqref{MAPpathdescription}.
	 Thus we obtain for any bounded and measurable functions $f,g$ and all $s,t\geq 0$ 
	\begin{align*}
	\EE_{J_0}\left[f\left(\frac{Z_{s+t}}{Z_s}\right)g(J_{s+t})|\mathcal{F}_s\right]
	&=\EE_{J_0}\left[f\left((-1)^{K_{s+t}-K_s} e^{-(Y_{s+t}-Y_s}\right)g(J_{s+t})\left.\right|\mathcal{F}_s\right]\\
	&= \EE_{J_0}\left[(f\circ h)\left( Y_{s+t}-Y_s, K_{s+t}-K_s \right)g(J_{s+t})\left. \right|\mathcal{F}_s\right]\\
	&= \EE_{J_s}\left[(f\circ h)\left( Y_{t},K_{t} \right)g(J_{t}) \right]\\
&=\EE_{J_s}\left[f\left((-1)^{K_t}e^{-Y_t}\right)g(J_t)\right]\\
&=\EE_{J_s}\left[f(Z_t)g(J_t)\right],
	\end{align*}
	with $h(x,y)=(-1)^y e^{-x}$. Hence  $(Z,J)$ is an MMP with respect to $\FF$ by Definition \ref{DefMMP}.\\	
	Conversely suppose $(Z,J)$ is an MMP with respect to $\FF$. Again define an auxiliary process $((\bar{Y},\bar{K}),J)$ by
	\begin{align*}
	\bar{Y}_t&:= - \log |Z_t|, \quad t\geq 0, \\
	\bar{K}_t&:= \sum_{0<s\leq t} \mathds{1}_{\{\bar{\Delta} Z_s<0\}}, \quad t\geq 0,
	\end{align*}
	where $\bar{\Delta}Z_s:=Z_s/Z_{s-}$ is a multiplicative increment of $Z$. Then clearly $Z_t=(-1)^{\bar{K}_t}e^{-\bar{Y}_t}$ a.s. for all $t\geq 0$. Further, we will prove in the following that $((\bar{Y},\bar{K}),J)$ is a bivariate $\FF$-MAP.\\
	Indeed, let $f,g$ be bounded and measurable functions, then for all $s,t\geq 0$ 
	\begin{align*}
	\EE_{J_0}[f(\bar{Y}_{s+t}-\bar{Y}_{s}, \bar{K}_{s+t}-\bar{K}_{s})g(J_{s+t})|\cF_s] &= \EE_{J_0}\Bigg[f\Bigg(-\log \left| \frac{Z_{s+t}}{Z_s} \right|, \sum_{s<u\leq s+t} \mathds{1}_{\{\bar{\Delta} Z_u<0\}} \Bigg) g(J_{s+t})|\cF_s \Bigg]\\
	&= \EE_{J_0}\Bigg[f\Bigg(-\log \left| \frac{Z_{s+t}}{Z_s} \right|, \sum_{s<u\leq s+t} \mathds{1}_{\big\{\frac{Z_u}{Z_s} \cdot \big(\frac{Z_{u-}}{Z_s}\big)^{-1} <0\big\}} \Bigg) g(J_{s+t})|\cF_s \Bigg]\\
	&= \EE_{J_s} \Bigg[f\Bigg(-\log \left| Z_{t} \right|, \sum_{0<u\leq t} \mathds{1}_{\{\bar{\Delta} Z_u<0\}} \Bigg) g(J_{t}) \Bigg]\\
	&= \EE_{J_s} [f (\bar{Y}_t , \bar{K}_t) g(J_{t})],
	\end{align*}
	where in the third line we have applied the MMP-property of $(Z,J)$ in the form given in Lemma \ref{lem-MMPproperty}. \\
	Further, define the process $(U,J)$ by setting
	\begin{equation}\label{UfromYK} U_t= -\bar{Y}_t-\frac{1}{2} [\bar{Y}^c,\bar{Y}^c]_t + \sum_{0<s\leq t} \Big(\Delta \bar{Y}_s - 1 + (-1)^{\Delta \bar{K}_s} e^{-\Delta \bar{Y}_s}\Big), \quad t\geq 0.\end{equation}
	As $((\bar{Y},\bar{K}),J)$ is a MAP, its additive component has a path decomposition \eqref{MAPpathdescription}. This implies that also $U$ has a path decomposition as given in \eqref{MAPpathdescription} and we may thus conclude that $(U,J)$ is an $\FF$-MAP. A simple computation using \eqref{exponentialexplicit} further shows that $Z=\cE(U)$. Finally, the fact that $U$ has no jumps of size $-1$ follows directly from \eqref{UfromYK}.	
\end{proof}

	\subsection{From Markov-modulated random recurrence equations to the \\Markov-modulated GOU process}\label{S2}
	
	We are now in the position to define the Markov-modulated GOU process as the continuous-time analogue of a solution to the \emph{Markov-modulated random recurrence equation (MMRRE)}
	\begin{align}\label{MMDifferenceEquation}
	Y_n=A_nY_{n-1}+B_n,\qquad n=1,2,\dots,
	\end{align}
	where $(A_n,B_n)_{n\geq 1}$ is a \emph{Markov-modulated sequence}, driven by a Markov chain $(M_n)_{n\in\NN_0}$ on $S$. Precisely, that means
	\begin{enumerate}[align=parleft, labelsep=1.6cm,leftmargin=*]
		\item[(MMRRE1)]
		$(A_1,B_1),(A_2,B_2),\dots$ are conditionally independent, given $M_0=j_0, M_1=j_1,\ldots$, for some $j_0,j_1,\ldots \in S$.
		\item[(MMRRE2)] The conditional law of $(A_n,B_n)$, given $M_0=j_0, M_1=j_1,\ldots$ for some $j_0,j_1,\ldots \in S$, depends only on $M_{n-1}$ and $M_n$ and is temporally homogeneous, i.e.
		\begin{align*}
		\PP\left((A_n,B_n)\in \cdot|M_{n-1}=j_{n-1}, M_n=j_n\right)=K_{j_{n-1},j_n}
		\end{align*}
		for a stochastic kernel $K$ from $S^2\to\RR^2$ and all $n\geq 1.$
		\item[(MMRRE3)] $Y_0$ is chosen conditionally independent of $((A_n,B_n),J_n)_{n\geq 1}$ given $J_0$.
	\end{enumerate}
Typically the Markov chain $(M_n)_{n\in\NN_0}$ is assumed to be ergodic to allow e.g. for a study of convergence of forward and backward iterations of MMRREs as done in \cite{ALSMEYER+BUCKMANN_StabilityPerpetuitiesMarkovianEnvironment}. For the results in this section however, ergodicity is not needed.

To derive a continuous time process $(V_t)_{t\geq 0}$ as an analogue to $(Y_n)_{n\in\NN}$, we require that for every $h>0$ the discretization $(V_{nh})_{n\geq 0}$ of $(V_t)_{t\geq 0}$ satisfies \eqref{MMDifferenceEquation} for $(A^{(h)}_n,B_n^{(h)})$ being a Markov-modulated sequence fulfilling (MMRRE1) - (MMRRE3) for some 
 Markov chain $(M_n^{(h)})_{n\in\NN}$ with state space $S$. Requiring that this divisibility property continues to hold for $h\to 0$ then implies that the process $(V_t)_{t\geq 0}$ has to satisfy
	\begin{align}\label{MAPGOUFunctionalEquation}
	V_t=A_{s,t}V_s+B_{s,t} \asure\quad\text{for all } 0\leq s\leq t,
	\end{align}
	for real-valued stochastic functionals $(A_{s,t}, B_{s,t})_{0\leq s\leq t}$ that are modulated by a
	 continuous-time Markov process $J=(J_t)_{t\geq 0}$ with state space $S$, such that the following properties hold (see also \cite{DEHAAN+KARANDIKAR_EmbeddingSDEcontinuoustimeprocess} for the original treatment):
	\begin{enumerate}[align=parleft, labelsep=1.6cm,leftmargin=*]
		\item[(MMGOU1)] Consistency: $A_{s,t}=A_{s,u}A_{u,t}$  and $B_{s,t}=A_{u,t}B_{s,u}+B_{u,t}$ a.s. for all $0\leq s\leq u\leq t$.
		\item[(MMGOU2)] Conditional independence: For $0\leq a\leq b\leq c\leq d$ the families $\lbrace (A_{s,t},B_{s,t})|a\leq s\leq t\leq b \rbrace$ and $\lbrace(A_{s,t},B_{s,t})|c\leq s\leq t\leq d \rbrace$ are conditionally independent given $J$.
		\item[(MMGOU3)] Conditional stationarity: Given $J$, the distribution of $(A_{s,s+h},B_{s,s+h})$ does not depend on $s$, i.e. $\PP_{J_0}((A_{s,s+h},B_{s,s+h})\in \cdot |\cF_s) = \PP_{J_s} ((A_{h},B_{h})\in \cdot)$ for all $h,s\geq 0$.
	\end{enumerate}
	In order to ensure c\`adl\`ag paths of the resulting process $(V_t)_{t\geq 0}$ we also make the following continuity assumption:
	\begin{enumerate}[align=parleft, labelsep=1.6cm,leftmargin=*]
		\item[(MMGOU4)] $A_t:=A_{0,t}\overset{\PP_j}\to 1$ and $B_t:=B_{0,t}\overset{\PP_j}\to 0 \text{ for }t\to 0,$ for all $j\in S$.
	\end{enumerate}
Finally, to avoid a degenerate behaviour of the resulting process we assume:
	\begin{enumerate}[align=parleft, labelsep=1.6cm,leftmargin=*]
\item[(MMGOU5)] $A_t\neq 0$ a.s. for all $t\geq 0$.
	\end{enumerate}
Clearly (MMGOU5) together with (MMGOU1) and (MMGOU3) implies $A_{s,t}\neq 0$ a.s. for all $0\leq s\leq t$.
	
We will now show that any process solving \eqref{MAPGOUFunctionalEquation} under the given conditions also solves an SDE driven by a bivariate MAP $((U,L),J)$. This is captured in Proposition \ref{PropMAPGOUSDE} below. Up to this end we give a representation of $(U,L)$ in terms of the stochastic functionals $(A,B)$ in the next lemma.

\begin{lem}\label{PropFunctionalequations xi U L}
	Suppose the functionals $(A_{s,t},B_{s,t})_{0\leq s\leq t}$ satisfy (MMGOU1)-(MMGOU5) for some 
	continuous-time Markov process $(J_t)_{t\geq 0}$ with state space $S$. 
	\begin{itemize}
		\item[a)] The process $(A_t,B_t)_{t\geq 0}:=(A_{0,t},B_{0,t})_{t\geq 0}$ has a càdlàg modification.
		\end{itemize}
	Furthermore, assuming that $(A_t,B_t)_{t\geq 0}$ is càdlàg, let $\FF$ be the natural augmented
	filtration of $(A_t , B_t )_{t\geq 0}$. Then 
		\begin{itemize}
		\item[b)] $(A_t, J_t)_{t\geq 0}$ is an $\FF$-MMP.
		\item[c)] $((U_t,L_t), J_t)_{t\geq 0}$ with 
		\begin{equation} \label{eqdefUL} \begin{pmatrix}
		U_t \\ L_t
		\end{pmatrix}:=\begin{pmatrix} \int_{(0,t]} A_{s-}^{-1}dA_s  \\ B_t-\int_{(0,t]} B_{s-}A_{s-}^{-1}dA_s \end{pmatrix}, \quad t\geq 0,\end{equation}
		is a bivariate $\FF$-MAP where the component $U$ has no jumps of size $-1$. 
		\item[d)]  The processes $(A_t)_{t\geq 0}$, $(B_t)_{t\geq 0}$, $(U_t)_{t\geq 0}$ and $(L_t)_{t\geq 0}$ are $\FF$-semimartingales.
	\end{itemize}
\end{lem}
\begin{proof}
	a) Using (MMGOU4) this can be shown in analogy to \cite[Lemma 2.1]{DEHAAN+KARANDIKAR_EmbeddingSDEcontinuoustimeprocess} via a variation of \cite[Thm. 14.20]{Breiman}.\\
	b) We check Definition \ref{DefMMP}, so let $f,g$ be  measurable and bounded, then by (MMGOU1) and (MMGOU3)
		\begin{align*}
		\EE_{J_0} \left[f\left(\left.\frac{A_{s+t}}{A_s}\right) g(J_{s+t})\right|\mathcal{F}_s\right]
		&=\EE_{J_0} \left[\left. f\left(\frac{A_{s}A_{s,s+t}}{A_{s}}\right) g(J_{s+t})\right|\mathcal{F}_s\right]\\
		&=\EE_{J_0} \left[f(A_{s,s+t}) g(J_{s+t})|\mathcal{F}_s\right]\\
		&=\EE_{J_s}\left[f\left(A_{t}\right)g(J_t)\right].
		\end{align*}
	c) As $U=\mathcal{L}og(A)$ it already follows from part b) together with Theorem \ref{1-1MAP-MMP} by uniqueness of the stochastic logarithm that $(U,J)$ is a MAP with respect to $\FF$ with $U$ having no jumps of size $-1$. However, to treat $(U,L)$ as a bivariate process, note that from (MMGOU1)
	$$U_{s+t}-U_s= \int_{(s,s+t]} A_{0,u-}^{-1} dA_{0,u} = \int_{(s,s+t]} A_{s,u-}^{-1} dA_{s,u}, \quad \text{a.s. for all }0\leq s\leq t.$$
	Moreover, via multiple applications of (MMGOU1)
	\begin{align*}
	L_{s+t}-L_s&=B_{0,s+t}-B_{0,s}-\int_{(s,s+t]} B_{0,u-}A_{0,u-}^{-1}dA_{0,u}\\
	&=\left(A_{s,s+t} B_{0,s} + B_{s,s+t}\right) -B_{0,s} -\int_{(s,s+t]} \left(A_{s,u-}B_{0,s}+B_{s,u-}\right) A_{0,u-}^{-1}dA_{0,u}\\
	&=B_{0,s}\left( A_{s,s+t}-1-A_{0,s}^{-1} \int_{(s,s+t]} dA_{0,u} \right)+B_{s,s+t}-\int_{(s,s+t]} B_{s,u-}A_{0,u-}^{-1}dA_{0,u}\\
	&=B_{0,s} \left( A_{s,s+t}-1 - A_{0,s}^{-1} \left( A_{0,s+t}-A_{0,s} \right)\right)+ B_{s,s+t}-\int_{(s, s+t]} B_{s,u-}A_{0,u-}^{-1}dA_{0,u}\\
	&=B_{s,s+t}-\int_{(s, s+t]} B_{s,u-}A_{s,u-}^{-1}dA_{s,u}, \quad \text{a.s. for all }0\leq s\leq t.
	\end{align*}
	Thus for $f,g$ measurable and bounded we get
	\begin{align*}
	\lefteqn{\EE_{J_0} \left[f\left(U_{s+t}-U_s, L_{s+t}-L_s\right) g(J_{s+t})|\cF_s\right]}\\
	&=\EE_{J_0} \left[ \left. f\left(\int_{(s,s+t]} A_{s,u-}^{-1} dA_{s,u}, B_{s,s+t}-\int_{(s,s+t]} B_{s,u-}A_{s,u-}^{-1}dA_{s,u}\right) g(J_{s+t})\right| \cF_s\right]\\
	&=\EE_{J_s}\left[f\left(\int_{(0,t]} A_{u-}^{-1} dA_{u} , B_t-\int_{(0,t]} B_{u-}A^{-1}_{u-}dA_u\right) g(J_t)\right]\\
	&=\EE_{J_s}\left[f(U_t,L_t)g(J_t)\right],
	\end{align*}
	where for the second equality we have used an extension of (MMGOU3) in the sense of Lemma \ref{lem-MAPproperty}. Hence $((U,L),J)$ is a MAP with respect to $\FF$.\\
	d) The semimartingale property of $(A_t)_{t\geq 0}$, $(U_t)_{t\geq 0}$ and $(L_t)_{t\geq 0}$ follows immediately since $(A_t, J_t)_{t\geq 0}$ is an $\FF$-MMP  and $((U,L),J)$ is an $\FF$-MAP. For $(B_t)_{t\geq 0}$ we observe from \eqref{eqdefUL} that it is the solution of the SDE
	\begin{equation}\label{eqSDEforB}
	dB_t= B_{t-} A_{t-}^{-1} dA_t + dL_t, \quad t\geq 0,
	\end{equation}
	which implies that also $(B_t)_{t\geq 0}$ is a semimartingale, see. e.g. \cite[Thm. V.7]{PROTTER_StochIntandSDE}.
\end{proof}

\begin{Prop}\label{PropMAPGOUSDE}
		Let $(A_{s,t},B_{s,t})_{0\leq s\leq t}$ satisfy (MMGOU1)-(MMGOU5) for some  
		continuous-time Markov process $(J_t)_{t\geq 0}$ with state space $S$. Then the process $(V_t)_{t\geq 0}$ given by \eqref{MAPGOUFunctionalEquation} is the unique solution to the SDE
	\begin{align}\label{MAPGOUSDE}
	dV_t=V_{t-}dU_t+dL_t,\quad t\geq 0,
	\end{align}
	for the MAP $((U_t,L_t),J_t)_{t\geq 0}$ defined in \eqref{eqdefUL}.
\end{Prop}
\begin{proof}
	From Lemma \ref{PropFunctionalequations xi U L}, Theorem \ref{1-1MAP-MMP} and the SDE \eqref{StochExpSDE}, it follows that  
	$$A_{t}=\mathcal{E}(U)_t=1+\int_{(0,t]} \mathcal{E}(U)_{s-}dU_s = 1+\int_{(0,t]} A_{s-} dU_s,$$ and in particular $dA_s=A_{s-}dU_s$. With this we obtain from \eqref{MAPGOUFunctionalEquation} and \eqref{eqdefUL}
	\begin{align*}
	V_t=V_0 A_t + B_t &= V_0 \left( 1+ \int_{(0,t]} A_{s-}dU_s\right) + L_t + \int_{(0,t]} B_{s-}A_{s-}^{-1} dA_s \\
	&=V_0 + V_0 \int_{(0,t]} A_{s-} dU_{s}+L_t+\int_{(0,t]} B_{s-}A_{s-}^{-1}A_{s-} dU_s\\
	&=V_0 + \int_{(0,t]} (V_0 A_{s-} +B_{s-}) dU_s+L_t\\
	&=V_0+ \int_{(0,t]} V_{s-}dU_s+L_t, \quad \text{a.s. }t\geq 0.
	\end{align*}
	Uniqueness follows from standard results on SDEs, see e.g. \cite[Thm. V.7]{PROTTER_StochIntandSDE}.
\end{proof}
	
In order to present the explicit solution of the SDE \eqref{MAPGOUSDE} in terms of the driving MAP $((U,L),J)$, we introduce a further bivariate MAP $((U, \eta),J)$ in the next lemma.
	
	\begin{lem}\label{PropExplicitetaB} Let $(A_{s,t},B_{s,t})_{0\leq s\leq t}$ satisfy (MMGOU1)-(MMGOU5) for some  
		continuous-time Markov process $(J_t)_{t\geq 0}$ with state space $S$, and assume that $(A_t)_{t\geq 0}$ and $(B_t)_{t\geq 0}$ are càdlàg. Let $\FF$ be the natural augmented filtration of $(A_t , B_t )_{t\geq 0}$, and let $((U_t,L_t),J_t)_{t\geq 0}$ be the resulting bivariate $\FF$-MAP defined in \eqref{eqdefUL}. 
		\begin{itemize}
			\item[a)] Set $H_t:=\int_{(0,t]} A_{s-} dA_s^{-1}=\mathcal{L}og(A^{-1})_t$, $t\geq 0$. Then $(H_t,J_t)_{t\geq 0}$ is a MAP with respect to $\FF$.  Hence $(H_t)_{t\geq 0}$ is an $\FF$-semimartingale that moreover admits the representation
			\begin{align*}
			H_t=-U_t+ \int_{(0,t]} \sigma^2_U (J_{s-}) ds + \sum_{0<s\leq t}\left(\Delta U_s - \frac{\Delta U_s}{1+\Delta U_s}\right), \quad t\geq 0.
			\end{align*}
			\item[b)] Let $\eta_t:=L_t+[H,L]_t$, $t\geq 0$, then $((U_t,\eta_t),J_t)_{t\geq 0}$ is a bivariate $\FF$-MAP. 
			\item[c)] The following relationships hold:
			\begin{align} 
			\eta_t&= L_t + \int_{(0,t]} \sigma_{U,L}(J_s)ds - \sum_{0<s\leq t} \frac{\Delta U_s \Delta L_s}{1+\Delta U_s}, \quad t\geq 0, \label{eq-etaviaUL} \\
			L_t&=\eta_t+[U,\eta]_t,\quad t\geq 0, \label{eq-LviaUeta}\\
			\text{and}\quad  B_t&=A_t \int_{(0,t]} A_{s-}^{-1} d\eta_s =\cE(U)_t \int_{(0,t]} \cE(U)_{s-}^{-1}  d\eta_s,\quad t\geq 0. \label{eq-BviaUeta}
			\end{align}
		\end{itemize}
	\end{lem}

\begin{remark} \label{rem-independencepurejumps}
	Note that in contrast to the L\'evy case considered in \cite{DEHAAN+KARANDIKAR_EmbeddingSDEcontinuoustimeprocess}, in the Markov-modulated setting independence of the switching L\'evy processes $U^j$ and $L^j$ for all $j\in S$ in general does not imply $L=\eta$ a.s. This is due to the fact that even if the pure switching parts of $U$ and $L$ are independent, the processes may have simultaneous jumps at times when $J$ jumps, and hence the jump part in \eqref{eq-etaviaUL} does not vanish. 
\end{remark}

\begin{proof}[Proof of Lemma \ref{PropExplicitetaB}]
	a) It can be checked in analogy to the proof of Lemma \ref{PropFunctionalequations xi U L} b) that $(A^{-1},J)$ is an $\FF$-MMP and by Theorem \ref{1-1MAP-MMP} this readily implies that $(H,J)$ is an $\FF$-MAP. The semimartingale property is thus immediate. To check the given representation for $(H_t)_{t\geq 0}$ use \eqref{exponentialexplicit} to see that $(H_t)_{t\geq 0}$ as given fulfils $\cE(H)_t=A_t^{-1}=\cE(U)_t^{-1}$. Uniqueness of the stochastic logarithm then yields the claim.\\
	b) We will prove that $((U,L, [H,L]),J)$ is an $\FF$-MAP to conclude that $((U,\eta),J)$ as given is an $\FF$-MAP as well. 
	By Lemma \ref{PropFunctionalequations xi U L} c) we already know that  $((U,L),J)$ is an $\FF$-MAP. 
	Further, by definition of the quadratic variation process and the representation of $(H_t)_{t\geq 0}$ developed in a) 
		\begin{align}
		[H,L]_t&=[H^c,L^c]_t+\sum_{0<s\leq t}\Delta H_s\Delta L_s \nonumber  =\int_{(0,t]} \sigma_{H,L}(J_s)ds + \sum_{0<s\leq t}\Delta H_s\Delta L_s \nonumber \\
		&= \int_{(0,t]} \sigma_{U,L}(J_s)ds - \sum_{0<s\leq t} \frac{\Delta U_s \Delta L_s}{1+\Delta U_s}, \quad t\geq 0. \label{eq-HUL2}
		\end{align}
		Thus for any $f,g$ bounded and measurable and any $s,t\geq 0$ we get
		\begin{align*}
		\lefteqn{\EE_{J_0} \left[ f\left(U_{s+t}-U_s, L_{s+t}-L_s, [H,L]_{s+t}-[H,L]_s\right) g(J_{t+s})|\mathcal{F}_s\right]}\\
		&=\EE_{J_0} \left[ f\left(U_{s+t}-U_s, L_{s+t}-L_s, [U^c,L^c]_{s+t} - [U^c,L^c]_s +\sum_{s<u\leq t+s} \frac{\Delta U_u \Delta L_u}{1+\Delta U_u}\right) g(J_{s+t})|\mathcal{F}_s\right]\\
		&=\EE_{J_s}\left[f\left(U_t, L_t, [U^c,L^c]_t +\sum_{0<u\leq t} \frac{\Delta U_u \Delta L_u}{1+\Delta U_u} \right)g(J_t)\right],
		\end{align*}
		by the MAP-property of $((U,L),J)$ in the sense of Lemma \ref{lem-MAPproperty}.\\
c) Equation \eqref{eq-etaviaUL} follows readily from \eqref{eq-HUL2} and the definition of $\eta$ in b). For \eqref{eq-LviaUeta} we observe that 
	\begin{align*}
	\eta_t+[U,\eta]_t=L_t+[H,L]_t+[U,L+[H,L]]_t=L_t+[H+U,L]_t+[U,[H,L]]_t, \quad t\geq 0,
	\end{align*}
	such that it remains to verify 
	\begin{equation} \label{eq-UHL1}
	[H+U,L]_t+[U,[H,L]]_t=0, \quad t\geq 0.
	\end{equation}
	Note that from a) it follows that $H_t+U_t=\int_{(0,t]} \sigma^2_U (J_{s}) ds + \sum_{0<s\leq t}\left(\Delta U_s - \frac{\Delta U_s}{1+\Delta U_s}\right)$ such that $H+U$ is of finite total variation. Moreover, $[H,L]$ is of finite variation by definition, see e.g. \cite[Cor. 1 of Thm. II.22]{PROTTER_StochIntandSDE}. Thus $[H+U,L]^c_t=0=[U,[H,L]]^c_t$ for all $t\geq 0$ and it suffices to consider the jumps. For these we obtain using the representation of $(H_t)_{t\geq 0}$ obtained in a) and \eqref{eq-HUL2}
	\begin{align*}
	\lefteqn{\sum_{0<s\leq t} (\Delta H_s + \Delta U_s) \Delta L_s + \sum_{0<s\leq t}\Delta U_s\Delta[H,L]_s}\\
	&= \sum_{0<s\leq t} \left( \Delta U_s - \frac{\Delta U_s}{1+\Delta U_s} \right) \Delta L_s - \sum_{0<s\leq t} \Delta U_s \frac{\Delta U_s \Delta L_s}{1+\Delta U_s}\\
	&=0, \quad t\geq 0,
	\end{align*}
	which implies \eqref{eq-UHL1} and hence \eqref{eq-LviaUeta}.\\ Finally, we check that $(B_t)_{t\geq0}$ as given in \eqref{eq-BviaUeta} solves the SDE \eqref{eqSDEforB}. Indeed, from \eqref{eqdefUL}, \eqref{eq-LviaUeta} and via integration by parts 
	\begin{align*}
	B_t&=L_t+\int_{(0,t]} B_{s-}A_{s-}^{-1}dA_s  =\eta_t+[U,\eta]_t+\int_{(0,t]} B_{s-}A_{s-}^{-1}dA_s\\
	&= \eta_t+[U,\eta]_t+ B_t A_t^{-1} A_t - \int_{(0,t]}A_{s-} d(B_s A_s^{-1}) - [B_\cdot A_\cdot^{-1}, A_\cdot]_t, \quad \text{a.s. for all } t\geq 0,
	\end{align*}
	which is equivalent to 
	\begin{equation} \label{eq-computeB} \eta_t- \int_{(0,t]}A_{s-} d(B_s A_s^{-1})+[U,\eta]_t- [B_\cdot A_\cdot^{-1}, A_\cdot]_t =0, \quad \text{a.s. for all }t\geq 0.
	\end{equation}
	Inserting the claimed formula $B_t=A_t \int_{(0,t]} A_{s-}^{-1}d\eta_s$, $t\geq 0$, gives
	\begin{align*}
	\int_{(0,t]}A_{s-} d(B_s A_s^{-1})+ [B_\cdot A_\cdot^{-1}, A_\cdot]_t  &= 	\int_{(0,t]}A_{s-} d\left( \int_{(0,s]} A_{u-}^{-1}d\eta_u \right) + \left[ \int_{(0,\cdot]} A_{s-}^{-1}d\eta_s , A_\cdot\right]_t \\ 
	& = \int_{(0,t]} d\eta_s  + \int_{(0,t]} A_{s-}^{-1} d[\eta, A]_s =\eta_t+  \left[\eta, \int_{(0,\cdot]} A_{s-}^{-1} dA_s \right]_t \\ &= \eta_t+ [\eta,U]_t, \quad \text{a.s. for all } t\geq 0,
	\end{align*}
	such that \eqref{eq-computeB} is fulfilled. Uniqueness of the solution to the SDE \eqref{eqSDEforB}, see e.g. \cite[Thm. V.7]{PROTTER_StochIntandSDE}, now yields the claim.
\end{proof}

	We may now summarize our findings in the main theorem of this section.
	
	\begin{theorem}\label{MainTheoremMAPGOU}
		Let $(A_{s,t},B_{s,t})_{0\leq s\leq t}$ satisfy (MMGOU1)-(MMGOU5) for a  
		continuous-time Markov process $(J_t)_{t\geq 0}$ with state space $S$, and assume that $(A_t)_{t\geq 0}$ and $(B_t)_{t\geq 0}$ are càdlàg. Let $\FF$ be the natural augmented filtration of $(A_t , B_t )_{t\geq 0}$. Then a bivariate $\FF$-MAP $((U,\eta),J)$ with $\Delta U\neq -1$ can be defined by
		\begin{equation}
		\begin{pmatrix}
		U_t \\ \eta_t
		\end{pmatrix} = \begin{pmatrix}
		\int_{(0,t]} A_{s-}^{-1} dA_s \\ \int_{(0,t]} A_{s-}d(B_sA_s^{-1})
		\end{pmatrix}, \quad t\geq 0.
		\end{equation}
		The process $(V_t)_{t\geq 0}$ defined in \eqref{MAPGOUFunctionalEquation} is then given by
		\begin{equation}\label{MAPGOUexplicit}
		V_t=\cE(U)_t\left(V_0 + \int_{(0,t]} \cE(U)_{s-}^{-1} d\eta_s\right), \quad t\geq 0,
		\end{equation}
		for some random variable $V_0$ that is conditionally independent of $((U_t,\eta_t),J_t)_{t\geq 0}$ given $J_0$. Moreover, the process $(V_t)_{t\geq 0}$ is the unique solution of the SDE \eqref{MAPGOUSDE} driven by the bivariate $\FF$-MAP $((U,L),J)$ with $(L_t)_{t\geq 0}$ given in
		\eqref{eq-LviaUeta}.\\
		Conversely, given a bivariate $\FF$-MAP $((U,\eta),J)$ such that $\Delta U\neq -1$,
	the functionals $(A_{s,t},B_{s,t})_{0\leq s\leq t}$ defined by
		\begin{equation}\label{eqdefAB}
	\begin{pmatrix}
	A_{s,t}\\ B_{s,t}
	\end{pmatrix} = \begin{pmatrix}
	\cE(U)_t \cE(U)_s^{-1} \\ \cE(U)_t\int_{(s,t]} \cE(U)_{u-}^{-1} d\eta_u
	\end{pmatrix}, \quad 0\leq s\leq t,
	\end{equation}
	fulfil the assumptions (MMGOU1) - (MMGOU5) with respect to the modulating Markov process $(J_t)_{t\geq 0}$.
	\end{theorem}

\begin{proof}[Proof of Theorem \ref{MainTheoremMAPGOU}]
Given $(A_{s,t},B_{s,t})_{0\leq s\leq t}$ it follows from Lemma \ref{PropExplicitetaB} b) that $((U,\eta),J)$ as given is an $\FF$-MAP. Hereby the specific form of $(U_t)_{t\geq 0}$ and the fact that $U$ has no jumps of size $-1$ have been shown in Lemma \ref{PropFunctionalequations xi U L}, while the specific form of $\eta$ follows from \eqref{eq-BviaUeta}. Moreover, as $U_t=\cL og(A)_t$ is equivalent to $A_t=\cE(U)_t$, we may insert this and \eqref{eq-BviaUeta} in \eqref{MAPGOUFunctionalEquation} to obtain \eqref{MAPGOUexplicit}. Finally, the fact that the given process $(V_t)_{t\geq 0}$ solves the SDE \eqref{MAPGOUSDE} with $L$ as given in \eqref{eq-LviaUeta} has been shown in Proposition~\ref{PropMAPGOUSDE}. \\
For the converse, we check that the functionals $(A_{s,t}, B_{s,t})_{0\leq s\leq t}$ as given in \eqref{eqdefAB} fulfil the assumptions (MMGOU1) to (MMGOU5). Indeed, consistency, i.e. (MMGOU1), follows by an immediate computation from \eqref{eqdefAB}, while conditional independence of the increments of $(A_{s,t}, B_{s,t})_{0\leq s\leq t}$, i.e.  (MMGOU2), follows from the same property of the MAP $((U,\eta),J)$ as shown e.g. in \cite[Thm. 2.22]{CINLAR_MAP21972}. To prove conditional stationarity we compute using \eqref{exponentialexplicit} and Lemma \ref{lem-MAPproperty}
\begin{align*}
\lefteqn{\PP_{J_0}((A_{s,s+h}, B_{s,s+h})\in \cdot |\cF_s) = \PP_{J_0}\Bigg(\Big(\cE(U)_{s+h} \cE(U)_s^{-1}, \cE(U)_{s+h} \int_{(s,s+h]} \cE(U)_u^{-1} d\eta_u\Big) \in \cdot \Bigg|\cF_s\Bigg)}\\
&= \PP_{J_0}\Bigg( \Big(\exp\big(U_{s+h}^c - U_s^c - \frac12 \left([U^c,U^c]_{s+h} - [U^c,U^c]_s\right)\big) \prod_{s<t\leq s+h} (1+\Delta U_t) ,  \\
&\qquad \qquad  \int_{(s,s+h]} \exp\big(U_{s+h}^c - U_u^c - \frac12 \left([U^c,U^c]_{s+h} - [U^c,U^c]_u\right)\big) \prod_{u<t\leq s+h} (1+\Delta U_t)d\eta_u\Big)
 \in \cdot \Bigg|\cF_s\Bigg)\\
 &= \PP_{J_s}\Bigg( \Big(\exp\big(U_{h}^c - \frac12 [U^c,U^c]_{h} \big) \prod_{0<t\leq h} (1+\Delta U_t) ,  \\
 &\qquad \qquad  \int_{(0,h]} \exp\big(U_{h}^c - U_u^c - \frac12 \left([U^c,U^c]_{h} - [U^c,U^c]_u\right)\big) \prod_{u<t\leq h} (1+\Delta U_t)d\eta_u\Big)
 \in \cdot \Bigg)\\
 &= \PP_{J_s} ((A_h, B_h)\in \cdot),
\end{align*}
which is (MMGOU3). Finally, (MMGOU4) follows directly from the c\`adl\`ag-paths of $(U,\eta)$, while (MMGOU5) is immediate from \eqref{exponentialexplicit} since $\Delta U \neq -1$. 
\end{proof}

		In \cite{DEHAAN+KARANDIKAR_EmbeddingSDEcontinuoustimeprocess}, instead of (MMGOU5), the authors imposed a positivity condition of the form 
			\begin{enumerate}[align=parleft, labelsep=1.7cm,leftmargin=*]
			\item[(MMGOU5')] $A_t>0$ for all $t\geq 0$ a.s.  
		\end{enumerate}
	that led to the well know form of the Lévy-driven GOU process \eqref{GOULevyexplizit}. A similar restriction can be used in the Markov-modulated setting and leads to the following variation of Theorem \ref{MainTheoremMAPGOU}.

	\begin{Prop}
		Let $(A_{s,t},B_{s,t})_{0\leq s\leq t}$ satisfy (MMGOU1)-(MMGOU4) and (MMGOU5') for a 
		 continuous-time Markov process $(J_t)_{t\geq 0}$ with  state space $S$, and assume that $(A_t)_{t\geq 0}$ and $(B_t)_{t\geq 0}$ are càdlàg. Let $\FF$ be the natural augmented filtration of $(A_t , B_t )_{t\geq 0}$. Then a bivariate $\FF$-MAP $((\xi,\eta),J)$ can be defined by
		\begin{equation}
		\begin{pmatrix}
		\xi_t \\ \eta_t
		\end{pmatrix} = \begin{pmatrix}
		- \log A_t \\ \int_{(0,t]} A_{s-}d(B_sA_s^{-1})
		\end{pmatrix}, \quad t\geq 0,
		\end{equation}
		such that the process $(V_t)_{t\geq 0}$ defined in \eqref{MAPGOUFunctionalEquation} is given by
		\begin{equation}\label{MAPGOUexplicitxi}
		V_t= e^{-\xi_t} \left(V_0 + \int_{(0,t]} e^{\xi_{s-}} d\eta_s\right), \quad t\geq 0,
		\end{equation}
		for some random variable $V_0$ that is conditionally independent of $((\xi_t,\eta_t),J_t)_{t\geq 0}$ given $J_0$. Moreover, the process $(V_t)_{t\geq 0}$ is the unique solution of the SDE \eqref{MAPGOUSDE} for the bivariate $\FF$-MAP $((U,L),J)$ defined via
		\begin{equation}\label{eq-ULviaxieta}
		\begin{pmatrix}
		U_t \\ L_t
		\end{pmatrix} = \begin{pmatrix}
		-\xi_t + \frac12 \int_{(0,t]} \sigma_\xi^2(J_s) ds + \sum_{0<s\leq t} \left( \Delta \xi_t + e^{-\Delta \xi_t} - 1 \right) \\ \eta_t -\int_{(0,t]} \sigma_{\xi,\eta}(J_s)ds +\sum_{0<s\leq t} \left(e^{-\Delta\xi_s}-1\right) \Delta\eta_s
		\end{pmatrix}, \quad t\geq 0,
		\end{equation}
		such that $\Delta U>-1$.\\
		Conversely, given a bivariate $\FF$-MAP $((\xi,\eta),J)$,
		the functionals $(A_{s,t},B_{s,t})_{0\leq s\leq t}$ defined by
		\begin{equation}\label{eqdefABxi}
		\begin{pmatrix}
		A_{s,t}\\ B_{s,t}
		\end{pmatrix} = \begin{pmatrix}
		e^{-(\xi_t-\xi_s)} \\ e^{-\xi_t} \int_{(s,t]} e^{\xi_{u-}} d\eta_u
		\end{pmatrix}, \quad 0\leq s\leq t,
		\end{equation}
		fulfil the assumptions (MMGOU1) - (MMGOU4) and (MMGOU5')  with respect to the modulating Markov process $(J_t)_{t\geq 0}$.
	\end{Prop}
\begin{proof}
	The proof is essentially the same as for Theorem \ref{MainTheoremMAPGOU} and we shall only shortly highlight the differences: 
	As $A_t>0$ a.s. for all $t\geq 0$ we may compute its natural logarithm instead of only the stochastic logarithm and consider $\xi_t=\log \cE(U)_t$ instead of the process $U_t$. The fact that $((\xi,L),J)$ is a bivariate MAP then follows by an analog computation as for $((U,L),J)$ in the proof of Lemma \ref{PropFunctionalequations xi U L} with $\xi$ being treated as in the computations leading to Theorem \ref{1-1MAP-MMP}. Further, as $A_t=e^{-\xi_t}$ and thus $B_t=e^{-\xi_t}\int_{(0,t]} e^{\xi_s-} d\eta_s$ by Lemma \ref{PropExplicitetaB} c), the given form of $(V_t)_{t\geq 0}$ in \eqref{MAPGOUexplicitxi} is immediate. Moreover it follows from Proposition  \ref{PropMAPGOUSDE} that $(V_t)_{t\geq 0}$ solves the SDE \eqref{MAPGOUSDE} for the MAP $((U,L),J)$, the precise formulas for $U$ and $L$ hereby can be shown by direct computations from $\xi_t=\log \cE(U)_t$ together with \eqref{exponentialexplicit} and \eqref{eq-LviaUeta}.\\
	The converse direction is a direct implication of the converse direction in Theorem \ref{MainTheoremMAPGOU} since $e^{-\xi_t}=\cE(U)_t$, $t\geq 0$.
\end{proof}

In analogy to the Lévy setting we now give the following definition.

\begin{defin}\rm \label{def-MMGOU} Let $((\xi_t,\eta_t), J_t)_{t\geq 0}$ be a bivariate $\FF$-MAP. Then the process $(V_t)_{t\geq 0}$, defined in \eqref{MAPGOUexplicitxi} for some random variable $V_0$ that is conditionally independent of $((\xi_t,\eta_t),J_t)_{t\geq 0}$ given $J_0$, will be called \textit{Markov-modulated generalized Ornstein-Uhlenbeck (MMGOU) process  driven by $((\xi,\eta),J)$}.
\end{defin}

\begin{ex}\label{ex-MMOU} The Markov-modulated Ornstein-Uhlenbeck  (MMOU) process considered e.g. in \cite{HUANG+SPREJ_MarkovmodulatedOUP2014} and \cite{LINDSKOG+MAJUMDER_ExactLongtimebehaviourMAP2019} is a special case of \eqref{MAPGOUexplicitxi} for $((\xi,\eta),J)_{t\geq 0}$ being a pure switching MAP with additive component given by
\begin{equation}\label{eq-xietaMMOU} \begin{pmatrix}\xi_t \\ \eta_t\end{pmatrix}=\begin{pmatrix} \int_{(0,t]}\gamma_\xi(J_s)ds \\ \int_{(0,t]}\gamma_\eta(J_s)ds + \int_{(0,t]} \sigma_\eta^2(J_s) dB_s \end{pmatrix}, \quad t\geq 0,\end{equation} 
for some standard Brownian motion $(B_t)_{t\geq 0}$, and Markovian component $(J_t)_{t\geq 0}$ defined on a finite state space $S$. In \cite{HUANG+SPREJ_MarkovmodulatedOUP2014} the authors develop the MMOU process as solution to an SDE of the form \eqref{MAPGOUSDE}, and - among other results - determine explicit expressions for its mean and variance, while in \cite{LINDSKOG+MAJUMDER_ExactLongtimebehaviourMAP2019} the long-time behaviour of the MMOU process is studied and an expression for its stationary distribution is derived, see also Example \ref{ex-statMMOU} below. \end{ex}

\begin{remark}
	One could also completely drop assumption (MMGOU5) (or (MMGOU5')) in the above derivations. The resulting continuous-time process would still be a solution of the SDE \eqref{MAPGOUSDE} for a bivariate MAP $((U,L),J)$, but $U$ could have jumps of size $-1$. In the Lévy setting such solutions have been considered e.g. in \cite{BEHME+LINDNER+MALLER_StationarySolutionsSDEGOU} and \cite{BLRR}. However, as the main focus of this section is a thorough introduction of the MMGOU process \eqref{MAPGOUexplicitxi}, we decided to exclude this case to avoid unnecessary further technicalities.
\end{remark}

\begin{remark} \label{rem-naive}
	A naive approach to define a Markov-modulated process of GOU type would have been to consider a concatenation of GOU processes, i.e. a stochastic process  $(\widetilde{V})_{t\geq 0}$ such that $\widetilde{V}_t$ behaves in law as a GOU process $V^j$ driven by a bivariate Lévy process $(\xi^j,\eta^j)$, whenever $J_t=j\in S$ and such that jumps of $J$ may induce additional jumps of $\tilde{V}$.\\
	If neither the MAP $((\xi_t,\eta_t),J_t)_{t\geq 0}$ nor the process $(\widetilde{V}_t)_{t\geq 0}$ exhibit jumps induced by the jumps of $(J_t)_{t\geq 0}$, the MMGOU process $(V_t)_{t\geq 0}$ driven by $((\xi,\eta),J)$ and $(\widetilde{V}_t)_{t\geq 0}$ are equal in law whenever $V_0\overset{d}=\widetilde{V}_0$.  However, as soon as additional jumps induced by the jumps of $J$ may occur, equality in distribution of $(V_t)_{t\geq 0}$ and $(\widetilde{V}_t)_{t\geq 0}$ only holds as long as $V_0\overset{d}=\widetilde{V}_0$ and one further chooses the additional jumps of $(\widetilde{V}_t)_{t\geq 0}$ such that
	$$\Delta \widetilde{V}_{T_n}\overset{d}= e^{-\Delta\xi_{T_{n}}}\Delta\eta_{T_{n}} + \left(e^{-\Delta \xi_{T_{n}}}-1\right) e^{-\xi_{T_{n}-}} \left( V_{0} + \int_{(0,T_{n})}e^{\xi_{s-}}d\eta_s \right).$$
	Without this assumption, the process $(\widetilde{V})_{t\geq 0}$ can not be expressed as solution to an SDE. Hence its further analysis is much more difficult than the subsequent analysis of the MMGOU process as defined above.
\end{remark}

	\section{Stationarity of the Markov-modulated GOU process}\label{S3}
	\setcounter{equation}{0}
	
	In this section we derive necessary and sufficient conditions for the existence of a strictly stationary MMGOU process. However, in order to present and prove our results in Section \ref{S3b} below, we first need to recall and develop further preliminary facts on Markov additive processes and on exponential functionals of Markov additive processes in the upcoming two subsections.
	
	\subsection{Further preliminaries on Markov additive processes}\label{S3a}
	
Consider a Markov additive process $(X,J)$ as introduced in Section \ref{S1a}. As before we assume the state space $S$ of $J$ to be at most countable and additionally from now on we will always assume $J$ to be ergodic, i.e. irreducible and positive recurrent. This ensures the existence of a stationary distribution of $J$ that we denote by $\pi=(\pi_j)_{j\in S}$. The terminology ``almost sure'' will in this section always be understood as $\PP_\pi$-almost sure with  $\PP_\pi(\cdot)=\sum_{j\in S}\pi_j \PP_j(\cdot)$. In particular $\PP_\pi$-almost sure is equivalent to $\PP_j$-almost sure for all $j\in S$.

 We write $\fatQ=(q_{ij})_{i,j\in S}$ for the intensity matrix of $J$ whose entries fulfil $q_{jj}<0$ and $\sum_{k\in S}q_{jk}=0$ for all $j\in S$. Further, under $\PP_j$, $j\in S$, we denote the \emph{return times} to state $j$ by 
$$\tau_0^{re}(j):=0,\quad \text{and} \quad \taure_n(j):=\inf\lbrace t>\tau_{n-1}^{re}(j): J_{t-}\neq J_t=j\rbrace, \, n\in\NN,$$
the \emph{exit times}, leaving state $j$, by 
$$\tauex_n(j):=\inf\lbrace t>\tau_{n-1}^{re}(j):J_t\neq j\rbrace, \, n\in\NN,$$ 
and the $n$-th \emph{sojourn time} as $$\mathbb{T}_j^n:=[\tau_{n-1}^{re}(j),\tauex_n(j)), \, n\in\NN.$$ 
The \emph{global sojourn time} in state $j$ as well as the sojourn time in $j$ up to time $t$ will be denoted by $$\mathbb{T}_j:=\bigcup_{n\geq 1}\mathbb{T}_j^n, \quad \text{and} \quad \mathbb{T}_j(t):= \TT_j \cap [0,t], \, t\geq 0.$$
 Since $J$ is Markovian, the sojourn time lengths  $\{|\mathbb{T}_j^n|, n\geq 1\}$ form an i.i.d. sequence of exponentially distributed random variables, more precisely   $$|\mathbb{T}_j^n|= \tauex_n(j) - \taure_{n-1}(j) \sim\Exp(-q_{jj}).$$ 
 The sequences of excursion time lengths and cycle lengths $\{\taure_n(j)-\tauex_n(j), n\geq 1\}$, $ \{ \taure_n(j)-\taure_{n-1}(j), n\geq 1\}$, and $\{ \tauex_n(j)-\tauex_{n-1}(j), n\geq 1\}$ are also i.i.d, but not exponentially distributed.  
 
 Fixing a state $j\in S$ we may reduce the time axis to $\mathbb{T}_j$ by ``conflating'' the excursions of the additive component $X$ for $t\not\in\TT_j$ to single jumps and identifying the $n$-th exit and $n$-th return time of $j$. The resulting \emph{conflated process} $(\hat{X}_t)_{t\geq 0}:=(\hat{X}^j_t)_{t\geq 0}$ as introduced in \cite{BEHME+SIDERIS_ExpFuncMAP2020} is then given by
 \begin{equation}\label{embeddedLevystructureMAP}\hat{X}^j_t:= X_{t+\sum_{k=1}^n (\taure_k(j)-\tauex_k(j))} \quad \text{for} \quad t\in \Big[\tauex_n(j) - \sum_{k=1}^{n-1} (\taure_k(j)-\tauex_k(j)), \tauex_{n+1}(j) - \sum_{k=1}^n (\taure_k(j)-\tauex_k(j))\Big), \end{equation}
 and as shown in \cite[Lemma 4.1]{BEHME+SIDERIS_ExpFuncMAP2020}, if $X$ is real-valued, then $\hat{X}$ is a L\'evy process with characteristic triplet $(\gamma_{X^{j}}, \sigma^2_{X^{j}}, \nu_{\hat{X}^j})$,
 where $\nu_{\hat{X}^j}(dx) = \nu_{X^{j}}(dx) - q_{jj}\PP_j(X_{\taure_1(j)}-X_{\tauex_1(j)}\in dx)$.
 
 We will also use the \emph{dual} MAP $(X^*,J^*)$ of $(X,J)$, sometimes also called its \emph{time-reversal}. 
Under the stationary distribution $\pi$ the MAP $(X^\ast,J^\ast)$ has intensity matrix $\fatQ^*=(q_{ij}^*)_{i,j\in S}=(\frac{\pi_j}{\pi_i} q_{ji})_{i,j\in S}$, and the additional jumps of $X^*$ induced by the jumps of $J$ are given by $\Phi_{X^*}^{ij}=-\Phi_{X}^{ji}$, while the switching Lévy character is kept with the Lévy process $X^{j,\ast}$ being equal in law to $-X^{j}$, cf. \cite[Appendix A.2]{DEREICH+DOERING+KYPRIANOU_RealselfsimilarprocessesStartedfromOrigin2017}. 
In particular it holds that
 	\begin{align}\label{DualityRelationMAPs}
 (X_{(t-s)-}-X_{t},J_{(t-s)-})_{0\leq s\leq t} \text{ under }\PP_\pi \text{ is equal in law to } (X^*_s,J^*_s)_{0\leq s\leq t}  \text{ under }\PP^\ast_\pi,
 \end{align}
 by  \cite[Lemma 21]{DEREICH+DOERING+KYPRIANOU_RealselfsimilarprocessesStartedfromOrigin2017},  with $\PP^\ast_j(\cdot)=\PP(\cdot|J_0^\ast=j)$ and $\PP_\pi^\ast$ defined accordingly.

	\subsection{Exponential integrals and exponential functionals of Markov additive processes}
	
		In Section \ref{S3b} below we will fully characterize stationarity of the MMGOU process driven by the bivariate MAP $((\xi,\eta),J)$.  As we will see, similar to the Lévy setting, stationarity properties of the MMGOU process are closely related to exponential functionals of the driving processes. We thus introduce for any bivariate MAP $((\zeta,\chi),J)$ the exponential integrals
	\begin{align*}
	\frace_{(\zeta,\chi)}(t)&:=\int_{(0,t]} e^{-\zeta_{u-}}d\chi_u, \quad t\geq 0,\\
	\text{and}\quad\fracef_{(\zeta,\chi)}(t)&:=e^{-\zeta_t}\int_{(0,t]} e^{\zeta_{u-}}d\chi_u, \quad t\geq 0,
	\end{align*}
	and in particular we write the \emph{exponential functional} of $((\zeta,\chi),J)$ (assuming existence of this limit) as \begin{equation}\label{MAPexpfunc} \mathfrak{E}_{(\zeta,\chi)}^\infty:= \int_{(0,\infty)} e^{-\zeta_{t-}}d\chi_t.\end{equation}

	The next lemma provides a useful relation between the exponential integral of the MAP $((\xi,\eta),J)$ driving the MMGOU process, and the exponential integral of the MAP $((\xi,L),J)$ for the process $L$ appearing in the SDE of the MMGOU process. Moreover, a relation between the two types of exponential integrals $\frace$ and $\fracef$ will be shown. The lemma is a generalization of \cite[Prop. 2.3]{LINDNER+MALLER_LevyintegralsStationarityGOUP} to the Markov-modulated situation. Note that in contrast to the Lévy setting treated in \cite{LINDNER+MALLER_LevyintegralsStationarityGOUP}, we have to use dual processes here to formulate the statement. Further note that the notion of time-reversion used in \cite{LINDNER+MALLER_LevyintegralsStationarityGOUP} for Lévy processes differs slightly from the notion of time-reversion used for MAPs introduced above.
	
	\begin{lem}\label{LindnerTh23MAPversion}
		Let $((\xi,\eta),J)$ be a bivariate MAP. Then $((\xi,L),J)$ with $(L_t)_{t\geq 0}$ defined in \eqref{eq-ULviaxieta} is a MAP. Moreover
		\begin{itemize}
			\item[a)] $\mathfrak{E}_{(\xi,L)}(t)=\mathfrak{E}_{(\xi,\eta)}(t)+[e^{-\xi},\eta]_t$ a.s. for all $t\geq 0$.
			\item[b)] $\mathfrak{E}_{(-\xi^*,-L^*)}(t)$ under $\PP_\pi^\ast$ is equal in law to $\mathfrak{F}_{(\xi,\eta)}(t)$ under $\PP_\pi$ for all $t\geq 0$.
		\end{itemize}
	\end{lem}
	\begin{proof}
		The path decomposition \eqref{MAPpathdescription} of $(\xi,\eta)$ implies a corresponding decomposition of $(\xi,L)$ via the definition of $(L_t)_{t\geq 0}$. Hence  $((\xi,L),J)$  is a MAP. \\
		For a) note that by Lemma \ref{PropExplicitetaB}c) the relation between $\eta$ and $L$ may be described by \eqref{eq-LviaUeta} with $(U_t)_{t\geq 0}$ as in \eqref{eq-ULviaxieta}. Thus we obtain
		\begin{align*}
		\FE_{(\xi,L)}(t)&=\int_{(0,t]} e^{-\xi_{s-}}dL_s=\int_{(0,t]} e^{-\xi_{s-}}d\left(\eta_s+[U,\eta]_s\right)=\FE_{(\xi,\eta)}(t)+\left[\int_{(0,\cdot]}  e^{-\xi_{s-}}dU_s,\eta\right]_t\\
		&= \FE_{(\xi,\eta)}(t)+\left[\int_{(0,\cdot]} \cE(U)_{s-} dU_s,\eta\right]_t = \FE_{(\xi,\eta)}(t)+\left[1+\cE(U),\eta\right]_t\\& = \FE_{(\xi,\eta)}(t)+\left[e^{-\xi},\eta\right]_t , \quad \text{a.s. for all } t\geq 0,
		\end{align*}
		where we have used the SDE of the stochastic exponential \eqref{StochExpSDE} in the second line.\\
		For b) define 
		$$((\xi^\dagger_s,\eta^\dagger_s), J^\dagger_s)_{0\leq s\leq t}:=(\xi_t-\xi_{(t-s)-},\eta_t-\eta_{(t-s)-}, J_{(t-s)-})_{0\leq s\leq t},$$
		and denote the completed natural filtration of $((\xi^\dagger_s, \eta^\dagger_s),J^\dagger_s)_{0\leq s\leq t}$ by $\HH$. Then $((\xi^\dagger_s, \eta^\dagger_s), J^\dagger_s)_{0\leq s\leq t}$ under $\PP_\pi$ is equal in law to $((-\xi_s^\ast, -\eta_s^\ast), J_s^\ast)_{0\leq s\leq t}$ under $\PP_\pi^\ast$. Further for any given càdlàg process $(D_s)_{0\leq s\leq t}$ define the time-reversed process $(\widetilde{D}_s)_{0\leq s\leq t}$ as 
		\begin{align*}
		\widetilde{D}_s:=\begin{cases}
		0,& s=0\\
		D_{(t-s)-}-D_{t-},& 0<s<t\\
		D_0-D_{t-},&s=t,
		\end{cases}
		\end{align*} 
		which is again a càdlàg process, see e.g. \cite[Section VI.4]{PROTTER_StochIntandSDE}.
		Then, by an application of \cite[Thm. VI.22]{PROTTER_StochIntandSDE}, we obtain in complete analogy to the proof of \cite[Lemma 6.1]{LINDNER+MALLER_LevyintegralsStationarityGOUP} in the Lévy case, that 
		\begin{align*}
		\widetilde{\mathfrak{E}_{(\xi^\dagger,\eta^\dagger)}}(u)+\widetilde{\left[e^{-\xi_\cdot^\dagger}, \eta^\dagger \right]}_u=\int_{(0,u]}  e^{-\xi^\dagger_{(t-s)-}}d\widetilde{\eta_s^\dagger},  \quad \text{a.s. for } 0\leq u\leq t,
		\end{align*}
		where the integrals are taken with respect to $\HH$. This implies, again in analogy to the proof of \cite[Lemma 6.1]{LINDNER+MALLER_LevyintegralsStationarityGOUP},
		\begin{align*}
		\int_{(0,t)}  e^{-\xi^\dagger_{s-}}d\eta_s^\dagger +\left[e^{-\xi^\dagger},\eta^\dagger\right]_{t-} = -\int_{(0,t]}  e^{-\xi_t+\xi_{s-}}d\widetilde{\eta^\dagger}_s= \int_{(0,t)} e^{-\xi_t+\xi_{s-}}d\eta_s =e^{-\xi_t}\int_{(0,t)} e^{\xi_{s-}}d\eta_s \quad \PP_\pi\text{-a.s.}
		\end{align*}
		As $\Delta\eta_t=\Delta\eta_t^\dagger=0$ a.s. for fixed $t\geq 0$, this yields that
		$$-\int_{(0,t]}  e^{\xi^\ast_{s-}}d\eta_s^\ast -\left[e^{\xi^\ast},\eta^\ast \right]_{t} \text{ under } \PP_\pi^\ast\text{ is equal in law to } e^{-\xi_t}\int_{(0,t]} e^{\xi_{s-}}d\eta_s\text{ under } \PP_\pi,$$
		and via a) we obtain b).
	\end{proof}

	Necessary and sufficient conditions for the convergence of the exponential integral $\frace(t)$ as $t\to\infty$ have been provided in \cite[Thm. 4.1]{BEHME+SIDERIS_ExpFuncMAP2020}. Using the relation between the integrals $\frace$ and $\fracef$ from Lemma \ref{LindnerTh23MAPversion}  we may now characterize convergence of the exponential integral $\fracef(t)$ as $t\to\infty$ in the following proposition.

	\begin{Prop}\label{ForwardConvergence}
		Let $((\xi,\eta),J)$ and $((\xi,L),J)$ with $(L_t)_{t\geq 0}$ defined in \eqref{eq-ULviaxieta} be bivariate MAPs. 
		Assume that $\lim_{t\in\mathbb{T}_j,t\to\infty}\xi_t=\infty$ $\PP_j$-a.s. and 
			\begin{equation}\label{eq-nesssuffweak}
		I_{(\xi,L)}^j:=\int_{(1,\infty)} \frac{\log q}{A^j_{\xi}(\log q)} |d\bar{\nu}_{L}^j(q)| <\infty,
			\end{equation}
	for some $j\in S$, where 
		\begin{align*} A^j_\xi(x)&:=\gamma_{\xi^j} + \nu_{\xi^j} ((1,\infty)) + \int_{(1,x)} \nu_{\xi^j} ((y,\infty)) dy \\ &\quad -q_{jj} \Bigg(\PP_j\big(\xi_{\taure_1(j)}-\xi_{\tauex_1(j)}\in (1,\infty)\big) + \int_{(1,x)} \PP_j\big(\xi_{\taure_1(j)}-\xi_{\tauex_1(j)}\in (y,\infty)\big) dy \Bigg),\end{align*}
	and 
	$$\bar{\nu}_L^j(dy):= \nu_{L^{j}}(dy) - q_{jj}\Bigg( \PP_j(L_{\taure_1(j)}- L_{\tauex_1(j)}\in dy) + \PP_j\Bigg( \int_{[\tauex_1(j),\taure_1(j)]} e^{-(\xi_{s-}-\xi_{\tauex_1(j)})} d L_s \in dy\Bigg)\Bigg),$$
	then $\frace_{(-\xi^\ast,-L^\ast)}(t)\to \frace_{(-\xi^\ast,-L^\ast)}^\infty$ in $\PP^\ast_j$-probability as $t\to\infty$ for all $j\in S$, and the integral $\fracef_{(\xi,\eta)}(t)$ converges in distribution  under $\PP_\pi$ with the limit being equal in law to $\frace_{(-\xi^\ast,-L^\ast)}^\infty$ under $\PP_\pi^\ast$ and thus independent of $j$.\\
	Conversely, if $\lim_{t\in\mathbb{T}_j,t\to\infty}\xi_t<\infty$ $\PP_j$-a.s. for all $j\in S$,  or if $I_{(\xi,L)}^j=\infty$ for all $j\in S$, then either there exists a sequence $\{c_j,j\in S\}$ in $\RR$ such that $\frace_{(-\xi^\ast,-L^\ast)}(t)= c_{J_0^\ast} -c_{J_t^\ast}e^{\xi_t^\ast}$ $\PP_\pi^\ast$-a.s. and
	$$\fracef_{(\xi,\eta)}(t) \text{ is equal in law to } c_{J_t}- c_{J_0} e^{-\xi_{t}} \text{ under } \PP_\pi,$$
or 	
	 $|\frace_{(-\xi^\ast,-L^\ast)}(t)|\overset{\PP^\ast_\pi}\longrightarrow \infty$ as $t\to\infty$  and $$|\fracef_{(\xi,\eta)}(t)|\overset{\PP_\pi}\to \infty \text{ as }t\to\infty.$$ 
	\end{Prop}

\begin{proof}
	By Lemma \ref{LindnerTh23MAPversion} for all $t\geq 0$  the integral $\mathfrak{F}_{(\xi,\eta)}(t)$ under $\PP_\pi$ is equal in law to  $\mathfrak{E}_{(-\xi^*,-L^*)}(t)$ under $\PP_\pi^\ast$. Moreover, by \cite[Thm. 4.1 or Prop. 4.12]{BEHME+SIDERIS_ExpFuncMAP2020} 
	if $\lim_{t\in \TT_j, t\to \infty} -\xi^\ast_t=\infty$ $\PP_j^\ast$-a.s. and 
	\begin{equation}\label{eq-nesssuffweakdual}
	\int_{(1,\infty)} \frac{\log q}{A^j_{-\xi^\ast}(\log q)} |d\bar{\nu}_{-L^\ast}^j(q)| <\infty,
	\end{equation}
	for some $j\in S$, then $\frace_{(-\xi^\ast,-L^\ast)}(t)\to \frace_{(-\xi^\ast,-L^\ast)}^\infty$ in $\PP^\ast_j$-probability as $t\to\infty$ for all $j\in S$. Hereby $\lim_{t\in \TT_j, t\to \infty} -\xi^\ast_t = \lim_{t\to \infty} -\hat{\xi}^\ast_t$, where the conflated process $\hat{\xi}^\ast$ is a Lévy process which  has the same law as $-\hat{\xi}$. Hence $\lim_{t\in \TT_j, t\to \infty} -\xi^\ast_t=\infty$ $\PP_j^\ast$-a.s. is equivalent to our assumption $\lim_{t\in \TT_j, t\to \infty} \xi_t=\infty$ $\PP_j$-a.s. Further from the definition of the dual process we see that $A^j_{-\xi^\ast}(x)=A^j_\xi (x)$ for all $x\geq 1$ and $\bar{\nu}_{-L^\ast}^j(\cdot) = \bar{\nu}_{L}^j(\cdot)$ such that \eqref{eq-nesssuffweakdual} is equivalent to \eqref{eq-nesssuffweak}. Hence from $\lim_{t\in \TT_j, t\to \infty} \xi_t=\infty$ $\PP_j$-a.s. and \eqref{eq-nesssuffweak} it follows that $\frace_{(-\xi^\ast,-L^\ast)}(t)\to \frace_{(-\xi^\ast,-L^\ast)}^\infty$ in $\PP^\ast_j$-probability as $t\to\infty$ for all $j\in S$, and thus this convergence holds in $\PP_\pi^\ast$-probability as well. Via Lemma \ref{LindnerTh23MAPversion} this implies that
	$\fracef_{(\xi,\eta)}(t)$ converges in distribution as $t\to\infty$ under $\PP_\pi$ and by uniqueness of the distributional limit we identify it with the distribution of $\frace_{(-\xi^\ast,-L^\ast)}^\infty$ under $\PP_\pi^\ast$.\\
	For the converse, by \cite[Thm. 4.1]{BEHME+SIDERIS_ExpFuncMAP2020}, if $\liminf_{t\in \TT_j, t\to \infty}\xi_t<\infty$ $\PP_j$-a.s. for all $j\in S$ (equivalently $\liminf_{t\in \TT_j, t\to \infty}-\xi^\ast_t<\infty$ $\PP_j^\ast$-a.s. for all $j\in S$) or if \eqref{eq-nesssuffweak} (equivalently \eqref{eq-nesssuffweakdual}) fails for all $j\in S$, then either there exists a sequence $\{c_j,j\in S\}$ in $\RR$ such that the exponential integral is degenerate in the sense that
	\begin{equation}\label{eq_deg00}
	\frace_{(-\xi^\ast,-L^\ast)}(t) = c_{J^\ast_0}- c_{J^\ast_t} e^{\xi^\ast_{t}} \quad \PP_{\pi}^\ast \text{-a.s.} 
	\end{equation}
	for all $t\geq 0$, or
	\begin{equation}\label{diverge}|\frace_{(-\xi^\ast,-L^\ast)}(t)|\overset{\PP_\pi^\ast}\longrightarrow \infty \quad \text{as }t\to\infty.\end{equation}	
	Clearly in case of \eqref{diverge} we may conclude via Lemma \ref{LindnerTh23MAPversion} that  $|\fracef_{(\xi,\eta)}(t)|\overset{d}\to \infty \text{ under } \PP_\pi \text{ as }t\to\infty,$
	from which we obtain for any $K>0$ that
	\begin{align*}
	\PP_\pi(|\fracef_{(\xi,\eta)}(t)|<K) &= \int_\RR \mathds{1}_{(-K,K)}(\fracef_{(\xi,\eta)}(t)) d\PP_\pi \underset{t\to\infty}\longrightarrow 0,
	\end{align*}
	i.e. 
	$$|\fracef_{(\xi,\eta)}(t)|\overset{\PP_\pi}\to \infty \text{ as }t\to\infty.$$
In the degenerate case \eqref{eq_deg00} the statement follows immediately, since
$c_{J^\ast_0}- c_{J^\ast_t} e^{\xi^\ast_{t}}$ under $\PP_\pi^\ast$ is equal in law to $c_{J_t}-c_{J_0}e^{-\xi_t}$ under $\PP_\pi$.
\end{proof}

	\subsection{Stationarity of the Markov-modulated GOU}\label{S3b}
	
We are now ready to state and prove the main result of this section which characterizes stationarity of the MMGOU process $(V_t)_{t\geq 0}$ and generalizes \cite[Thm. 2.1]{LINDNER+MALLER_LevyintegralsStationarityGOUP} to the Markov-modulated setting.
	
\begin{theorem}\label{StationarityMAPGOU}
		Let $(V_t)_{t\geq 0}$ be the MMGOU process driven by the bivariate MAP $((\xi,\eta),J)$ and recall that  $((\xi,L),J)$ with $(L_t)_{t\geq 0}$ defined in \eqref{eq-ULviaxieta} is a MAP.\\
		If $(V_t)_{t\geq 0}$ is strictly stationary, then one of the following two conditions is satisfied:
		\begin{itemize}
			\item[a)] The exponential integral $\mathfrak{E}_{(-\xi^*,-L^*)}(t)$ converges in $\PP_\pi^\ast$-probability to some proper random variable for $t\to\infty$; equivalently 
			$\lim_{t\in\mathbb{T}_j,t\to\infty}\xi_t=\infty$ $\PP_j$-a.s. and $I_{(\xi,L)}^j<\infty$ hold for some $j\in S$.
			\item[b)] There exists a sequence of constants $\{c_j, j\in S\}$ such that $V_t=c_{J_t}$ $\PP_\pi$-a.s. for all $t\geq 0$. In particular in this case $\PP_\pi(V_t=c_j)=\pi_j$ for all $t\in\mathbb{R}$, thus $V_t\iv V_0$.
		\end{itemize}
	Conversely, if a) or b) holds, then there exists a finite random variable $V_\infty$, such that $(V_t)_{t\geq 0}$ started with $V_0\iv V_\infty$ is strictly stationary. In case of a) the stationary random variable $V_\infty$ under $\PP_\pi$ is equal in law to $\mathfrak{E}_{(-\xi^*,-L^*)}^{\infty}$ under $\PP_\pi^\ast$.
	\end{theorem}

Before we prove Theorem \ref{StationarityMAPGOU}, let us mention that the the Markov-modulated situation exhibits several differences to the case of $|S|=1$, i.e. when $(\xi,\eta)$ and $(\xi,L)$ are Lévy processes as it has been studied in \cite{LINDNER+MALLER_LevyintegralsStationarityGOUP}. Some of them will be highlighted in the following remarks.

\begin{remarks}
	\begin{enumerate}
		\item As in the Lévy case,  existence and the distribution of a strictly stationary solution of \eqref{MAPGOUSDE} is closely related to convergence of an exponential functional. However, in the Markov-modulated situation this exponential functional is driven by the dual process $((-\xi^*,-L^*),J^*)$, instead of the $(\xi,L)$ in the L\'evy setting. This agrees with a similar phenomenon previously observed in discrete time: While forward and backward iterations of random recurrence equations with i.i.d. coefficients lead to the same distribution, see \cite{GOLDIE+MALLER_StabilityPerpetuities2000}, this is no longer true for their Markov-modulated counterpart, i.e. for MMRREs, see \cite{ALSMEYER+BUCKMANN_StabilityPerpetuitiesMarkovianEnvironment}.
		\item Whenever $|S|$ is finite (which is in particular true in the L\'evy setting $|S|=1$), it has been pointed out in \cite[Rem. 4.2]{BEHME+SIDERIS_ExpFuncMAP2020} that convergence in $\PP_j$-probability of $\FE_{(-\xi^*,-L^*)}(t)$ as $t\to \infty$ is equivalent to $\PP_\pi$-almost sure convergence of $\FE_{(-\xi^*,-L^*)}(t)$ as long as degeneracy of the integral is ruled out. In general however this is not the case as illustrated in \cite[Ex. 4.4]{BEHME+SIDERIS_ExpFuncMAP2020}.
		\item In the Markov-modulated setting it is not necessary to assume $\xi_t\to\infty$ $\PP_\pi$-a.s. to prove convergence of $\FE_{(-\xi^*,-L^*)}(t)$ as $t\to \infty$. An example that demonstrates this specific behaviour of exponential functionals of MAPs is given by \cite[Ex. 4.3]{BEHME+SIDERIS_ExpFuncMAP2020}. Thus $\xi_t\to\infty$ $\PP_\pi$-a.s. is in general not necessary for strict stationarity of the MMGOU process.
		\item While the degenerate case in the L\'evy case yields an a.s. constant process, degeneracy in the Markov-modulated setting allows for more flexibility as the stationary MMGOU process may switch between the different constants $c_j, j\in S$. \\However, this will only happen if the additive component of the MAP $((\xi,\eta),J)$ does exhibit jumps induced by the jumps of $J$. Indeed, as it has been observed in \cite[Rem. 4.8]{BEHME+SIDERIS_ExpFuncMAP2020}, assuming $((\xi,\eta),J)$ to be a pure switching process (i.e. $(\xi^{(2)},\eta^{(2)})=(0,0)$ a.s. in the decomposition \eqref{MAPpathdescription}), we may conclude from \eqref{degenerate} that $c_j=c$ $\forall j\in S$ for some $c\in \RR$.\\
		Further note that the sequence $\{c_j,j\in S\}$ can only be uniquely determined, when $\xi$ is assumed to be not \emph{null-homologue}, i.e. if $\PP_j\left(\hat{\xi}^j_t=0\right)<1$ for all $j\in S$, as it was assumed in \cite{BEHME+SIDERIS_ExpFuncMAP2020}. Since we did not exclude null-homology here, we may have some $j\in S$ such that
		\begin{align*}
		\PP_j\left(\FE_{(-\xi^*,-L^*)}(t)=c_j\left(1-e^{\xi^*_t}\right) \Big| t\in \TT_j\right)=\PP_j\left(\FE_{(\xi^*,L^*)}(t)=0|t\in\TT^j\right)=1,
		\end{align*}
		and in this case $c_j$ can not be determined. 
	\end{enumerate}
\end{remarks}

\begin{proof}[Proof of Theorem \ref{StationarityMAPGOU}]
	We extend the proof of \cite[Thm. 2.1]{LINDNER+MALLER_LevyintegralsStationarityGOUP} to the Markov-modulated situation.\\ Suppose that $(V_t)_{t\geq 0}$ as defined in \eqref{MAPGOUexplicitxi} is strictly stationary under $\PP_\pi$. Then $V_t\iv V_0$ for all $t\geq 0$ and in particular $V_t\inV V_0$ as $t\to\infty$. 
	Fix $j\in S$ and assume first that $\lim_{t\in\mathbb{T}_j,t\to\infty}\xi_t=\infty$ $\PP_j$-a.s. Then by \cite[Lemma 4.10]{BEHME+SIDERIS_ExpFuncMAP2020} it holds $\lim_{t\in\mathbb{T}_j,t\to\infty}\xi_t=\infty$ $\PP_j$-a.s. for all $j\in S$. Hence we may conclude that $e^{-\xi_t}\inPpi 0$ as $t\to\infty$. Since $(V_t)_{t\geq 0}$ given by $$V_t=e^{-\xi_t}V_0 + e^{-\xi_t}\int_{(0,t]} e^{\xi_{s-}}d\eta_s=e^{-\xi_t}V_0 + \fracef_{(\xi,\eta)}(t), \quad t\geq 0,$$  is assumed to be strictly stationary, by Slutsky's theorem we must have that $\mathfrak{F}_{(\xi,\eta)}(t)$ converges in distribution to $V_0$ as $t\to\infty$. However by Proposition \ref{ForwardConvergence} this is only possible, if the conditions given in a) are fulfilled. \\
	Now assume that  $\lim_{t\in\mathbb{T}_j,t\to\infty}\xi_t<\infty$ $\PP_j$-a.s. for all $j\in S$, then by Proposition \ref{ForwardConvergence} two possible cases can occur. First, there may exist a  sequence of real constants $\{c_j, j\in S\}$ such that $\frace_{(-\xi^*,-L^*)}(t)=c_{J^\ast_0}+c_{J^\ast_t}e^{\xi^\ast_t}$ $\PP^\ast_\pi$-a.s. By \cite[Prop. 4.7]{BEHME+SIDERIS_ExpFuncMAP2020} this implies that
	$$-L_t^\ast = - \int_{(0,t]} c_{J_{s-}^\ast} e^{-\xi_{s-}^\ast} d(e^{\xi_s^\ast}) - \int_{(0,t]} dc_{J_s^\ast}\quad \PP_\pi^\ast\text{-a.s.}$$
	As dependencies are kept by time-reverting $((\xi^\ast,L^\ast),J^\ast)$ we conclude that 
	\begin{equation}\label{degenerate} L_t = - \int_{(0,t]} c_{J_{s-}} e^{-\xi_{s-}} d(e^{\xi_s}) - \int_{(0,t]} dc_{J_s} = - \int_{(0,t]} c_{J_{s-}} dU_s  - \int_{(0,t]} dc_{J_s}\quad \PP_\pi\text{-a.s.},\end{equation}
	with $(U_t)_{t\geq 0}$ as in \eqref{eq-ULviaxieta} such that $U_t=\cL og (e^{-\xi_t})$, $t\geq 0$. Hence $(V_t)_{t\geq 0}$ is a strictly stationary solution to the SDE
	$$dV_t=V_{t-} dU_t + dL_t = (V_{t-}-c_{J_{t-}})dU_t - dc_{J_t}, \quad t\geq 0,$$
	which is obviously solved by $V_t=c_{J_t}$ $\PP_\pi$-a.s. As $(c_{J_t})_{t\geq 0}$ is indeed strictly stationary under $\PP_\pi$ this implies b).\\
	Second, assume that $\lim_{t\in\mathbb{T}_j,t\to\infty}\xi_t<\infty$ $\PP_j$-a.s. for all $j\in S$ and $|\frace_{(-\xi^*,-L^*)}(t)|\inPpi\infty$ as $t\to\infty$  such that $|\fracef_{(\xi,\eta)}(t)|\overset{\PP_\pi}\to \infty$ as $t\to\infty$. Then since $V_t\overset{d}\to V_0$ as $t\to \infty$ under $\PP_\pi$ and because $V_0$ is finite, by \eqref{MAPGOUexplicitxi} and Slutsky's theorem we conclude that $|V_0|e^{-\xi_t}\inPpi\infty$, i.e. $\xi_t\inPpi - \infty$ as $t\to\infty$. Thus 
	\begin{align} \label{stopherenoncausal}
	V_0+\int_{(0,t]} e^{\xi_{s-}}d\eta_s=e^{\xi_t}V_t\overset{\PP_\pi}\longrightarrow 0 \quad \text{as} \quad t\to\infty,
	\end{align}
	since $V_t$ converges in distribution to a finite random variable. Moreover $\int_{(0,t]} e^{\xi_{s-}}d\eta_s$ is conditionally independent of $V_0$ given $J_0$ and thus its limit in probability $-V_0$ is also conditionally independent of $V_0$ given $J_0$. Hence $V_0$ is constant given $J_0$, which implies $V_t=c_{J_t}$ for a sequence of real constants $\{c_j, j\in S\}$, i.e. b).\\
	For the converse assume first that a) holds. Then by \cite[Lemma 4.10]{BEHME+SIDERIS_ExpFuncMAP2020} we have that $\lim_{t\in\mathbb{T}_j,t\to\infty}\xi_t=\infty$ $\PP_j$-a.s. for all $j\in S$ and hence $e^{-\xi_t} V_0 \inPpi 0$ as $t\to \infty$. Further, by Proposition  \ref{ForwardConvergence} the integral $\mathfrak{F}_{(\xi,\eta)}(t)$ converges in distribution as $t\to\infty$ with the limit being equal in law to $\mathfrak{E}_{(-\xi^*,-L^*)}^{\infty}$ under $\PP_\pi^\ast$. 
	Thus by \eqref{MAPGOUexplicitxi} and Slutsky's theorem we obtain that also $V_t$ converges in distribution as $t\to\infty$ with the limit being equal in law to $\mathfrak{E}_{(-\xi^*,-L^*)}^{\infty}$ under $\PP_\pi^\ast$. Now choose $V_0$ under $\PP_\pi$ conditionally independent of $((\xi_t,\eta_t),J_t)_{t\geq 0}$ given $J_0$ and equal in law to $\mathfrak{E}_{(-\xi^*,-L^*)}^{\infty}$ under $\PP_\pi^\ast$. From Section \ref{S2} we know that	
	$$V_t=A_{t-h,t} V_{t-h} + B_{t-h,t}, \quad  \text{a.s. for all } 0\leq h\leq t, $$
	with $(A_{s,t},B_{s,t})_{0\leq s\leq t}$ conditionally stationary by (MMGOU3). Thus letting $h$ fixed and $t\to \infty$ this yields under $\PP_\pi$
	$$V_0\overset{d}= A_{0,h} V_0 + B_{0,h} = V_h$$
	and hence $V_t\overset{d}=V_0$ under $\PP_\pi$ for all $t\geq 0$. As $(V_t, J_t)_{t\geq 0}$ is a time-homogeneous Markov process this implies strict stationarity of $(V_t)_{t\geq 0}$ under $\PP_\pi$.\\
	In case of b) stationarity is immediate.
\end{proof}

		\begin{ex} \label{ex-statMMOU} Consider the MMOU process as in Example \ref{ex-MMOU}, i.e. the MMGOU process $(V_t)_{t\geq 0}$ driven by $((\xi,\eta),J)$ with $(\xi,\eta)$ specified in \eqref{eq-xietaMMOU}. For this process stationarity has been considered already in  \cite{ZHANG+WANG_StationarydistributionOUP2stateMarkovswitching2017} for $|S|=2$, and in \cite{LINDSKOG+MAJUMDER_ExactLongtimebehaviourMAP2019} for countable $S$. The very special setting chosen in \cite{ZHANG+WANG_StationarydistributionOUP2stateMarkovswitching2017} allows the authors to describe the stationary distribution of the MMOU process via Fourier methods. In \cite{LINDSKOG+MAJUMDER_ExactLongtimebehaviourMAP2019} the stationary distribution of the MMOU process (and even of a "L\'evy-driven MMOU process") is determined to be the distribution of $\sum_{j\in S} \pi_j S_j$ for some random variables $S_j$ that are described by exponential integrals over the excursions of $J$. The method applied in  \cite{LINDSKOG+MAJUMDER_ExactLongtimebehaviourMAP2019} is based on focusing on return times of $(J_t)_{t\geq 0}$ which leads to a random recurrence equation with i.i.d. coefficients for $(V_t)_{t\geq 0}$. This approach could have been extended to our setting as well. However the representation of the stationary distribution obtained in \cite[Thm. 1 and Remark 3.3]{LINDSKOG+MAJUMDER_ExactLongtimebehaviourMAP2019} seems less intuitive to us compared to the representation as exponential functional. \\
			Indeed, applying the above Theorem \ref{StationarityMAPGOU} on the MMOU process, we immediately observe that option b) is impossible as long as $\sigma^2_\eta(j)\neq 0$ for some $j\in S$. Thus the process $(V_t)_{t\geq 0}$ can not be degenerate under this standard condition. Hence  $V_0$ can be chosen conditionally independent of $((\xi_t, \eta_t),J_t)_{t\geq 0}$ given $J_0$ and such that $(V_t)_{t\geq 0}$ is strictly stationary if and only if the exponential functional $\frace_{(-\xi^\ast,-L^\ast)}(t)$ converges in $\PP_\pi^\ast$-probability. However, as jumps of $J$ do not induce additional jumps of $(\xi,\eta)$ in the present situation and $\xi$ and $\eta$ are independent, we observe from \eqref{eq-ULviaxieta} that $L_t=\eta_t$, $t\geq 0$. Thus the MAP $((-\xi^\ast_t, -L^\ast_t),J_t)_{t\geq 0}$ can be represented as 
		$$ \begin{pmatrix}-\xi_t^\ast \\ -L_t^\ast \end{pmatrix}=\begin{pmatrix} \int_{(0,t]}\gamma_\xi(J^\ast_s)ds \\ \int_{(0,t]}\gamma_\eta(J^\ast_s)ds + \int_{(0,t]} \sigma_\eta^2(J_s^\ast) dB^\ast_s \end{pmatrix} $$
		for some standard Brownian motion $(B_t^\ast)_{t\geq 0}$.\\ A simple sufficient condition for existence of a stationary solution is given by
		$$0<\sum_{j\in S} \pi_j \gamma_\xi(j)<\infty, \quad \text{and} \quad \sup_{j\in S}(|\gamma_\eta(j)|+\sigma_\eta^2(j))<\infty,$$
		see \cite[Prop. 5.2]{BEHME+SIDERIS_ExpFuncMAP2020}, while the necessary and sufficient condition can be derived from Prop. \ref{ForwardConvergence} or \cite[Thm. 4.1]{BEHME+SIDERIS_ExpFuncMAP2020}. 
	The stationary distribution of the MMOU process in this case is given by the distribution of
		$$\frace_{(-\xi^*,-L^*)}^\infty = \int_{(0,\infty)}e^{-\int_{(0,t]}\gamma_\xi(J^\ast_s)ds} \gamma_\eta(J^\ast_t)dt + \int_{(0,\infty)} e^{-\int_{(0,t]}\gamma_\xi(J^\ast_s)ds}\sigma_\eta^2(J_t^\ast) dB^\ast_t.$$
			\end{ex}\bigskip
		
We end this section with a short discussion on \emph{non-causal} MMGOU processes, i.e. MMGOU processes $(V_t)_{t\geq 0}$ driven by a bivariate MAP $((\xi,\eta),J)$, where $V_0$ is not assumed to be conditionally independent of $((\xi_t,\eta_t),J_t)_{t\geq 0}$ given $J_0$. In the L\'evy setting such solutions have been considered in \cite{BEHME+LINDNER+MALLER_StationarySolutionsSDEGOU} and as long as $S$ is assumed to be finite, it is not surprising that a similar statement as \cite[Thm. 2.1]{BEHME+LINDNER+MALLER_StationarySolutionsSDEGOU} can be shown as follows.

\begin{Prop}
		Let $(V_t)_{t\geq 0}$ be the MMGOU process driven by the bivariate $\FF$-MAP $((\xi,\eta),J)$ and recall that  $((\xi,L),J)$ with $(L_t)_{t\geq 0}$ defined in \eqref{eq-ULviaxieta} is an $\FF$-MAP. Assume that $(V_t)_{t\geq 0}$ may be non-causal and additionally assume $|S|<\infty$. \\
	If $(V_t)_{t\geq 0}$ is strictly stationary, then either a) or b) from Theorem \ref{StationarityMAPGOU} hold, or,
	\begin{itemize}
		\item[c)] The exponential integral $\mathfrak{E}_{(-\xi,\eta)}(t)$ converges $\PP_\pi$-a.s. to some proper random variable for $t\to\infty$; equivalently 
		$\lim_{t\in\mathbb{T}_j,t\to\infty}\xi_t=-\infty$ $\PP_j$-a.s. and $I_{(-\xi,\eta)}^j<\infty$ hold for some $j\in S$.
	\end{itemize}
Conversely, if a), b), or c) holds, then there exists a finite random variable $V_0$, such that $(V_t)_{t\geq 0}$ started with $V_0$ is strictly stationary. In case of a) the stationary process is obtained by choosing $V_0$ conditionally independent of $((\xi_t,\eta_t),J_t)_{t\geq 0}$ given $J_0$ and equal in law to $\mathfrak{E}_{(-\xi^*,-L^*)}^{\infty}$ under $\PP_\pi^\ast$. In case of c) the stationary process is obtained by choosing $V_0=\int_{(0,\infty)} e^{\xi_{s-}}d\eta_s \in \cF_\infty$  a.s., and the stationary process is given by 
$V_t= -e^{-\xi_t} \int_{(t,\infty)} e^{\xi_{s-}}d\eta_s$ a.s. for all $t\geq 0$.
\end{Prop}

\begin{proof}
	We follow the proof of Theorem \ref{StationarityMAPGOU} up to Equation \eqref{stopherenoncausal} and note that we have not made use of the conditional independence assumption when proving that a) or b) may occur in the treated cases. \\
	It remains to reconsider the case that $\lim_{t\in\mathbb{T}_j,t\to\infty}\xi_t<\infty$ $\PP_j$-a.s. for all $j\in S$ and $|\frace_{(-\xi^*,-L^*)}(t)|\inPpi\infty$ as $t\to\infty$. As we already argued, $\xi_t\inPpi - \infty$ as $t\to\infty$ and \eqref{stopherenoncausal} holds. From \eqref{stopherenoncausal} we see immediately that stationarity implies that we have to chose $V_0$ as limit in $\PP_\pi$-probability of $\int_{(0,t]}e^{\xi_{s-}}d\eta_s = \frace_{(-\xi,\eta)}(t)$, and in particular $\frace_{(-\xi,\eta)}(t)$ has to converge in $\PP_\pi$-probability. As seen in the proof of Theorem \ref{StationarityMAPGOU} a degenerate behaviour of $\frace_{(-\xi,\eta)}(t)$ is possible in this case, leading to b) again. Thus assume the exponential integral $\frace_{(-\xi,\eta)}(t)$ is not degenerate. 
	Then, as $|S|<\infty$, by \cite[Rem. 4.2]{BEHME+SIDERIS_ExpFuncMAP2020} convergence in $\PP_\pi$-probability is equivalent to $\PP_\pi$-almost sure convergence of $\frace_{(-\xi,\eta)}$ which yields c). \\
	For the converse, assuming a) or b) we may again follow the proof of Theorem \ref{StationarityMAPGOU}. Assuming c) and setting $V_0=-\int_{(0,\infty)} e^{\xi_{s-}}d\eta_s$ $\PP_\pi$-a.s. clearly implies $V_t=\int_{(t,\infty)} e^{\xi_{s-}-\xi_t}d\eta_s$ $\PP_\pi$-a.s. This process is, conditionally given $J$, strictly stationary, since $(\xi,\eta)$ has conditionally stationary increments given $J$ by Lemma \ref{lem-MAPproperty}.
\end{proof}

\begin{remark}
Unfortunately, assuming $S$ to be countably infinite in the above theorem, we cannot determine the explicit form of the non-causal stationary distribution, since convergence in probability does not imply convergence in an almost sure sense of the exponential functional. However, the non-causal stationary solution that we obtained still exists for $|S|=\infty$ as long as the exponential functional $\frace_{(-\xi,\eta)}(t)$ converges almost surely.
\end{remark}

\begin{remark}
As it has been discussed also in \cite{BEHME+LINDNER+MALLER_StationarySolutionsSDEGOU}, it is in general not clear, whether the non-causal process $(V_t)_{t\geq 0}$ still satisfies the SDE \eqref{MAPGOUSDE}. This is due to the fact that $(U,L)$ may not be a semimartingale with respect to the enlarged filtration $\GG$, which is the smallest filtration containing $\FF$, satisfying the usual hypothesis, and such that $V_0$ is $\cG_0$-measurable. 
However, following the lines of \cite[Thm. 5.1]{BEHME+LINDNER+MALLER_StationarySolutionsSDEGOU} or \cite[Thm. 3.6]{Jacodfiltration}, one can show that whenever the distribution of $V_0$ is the Dirac measure or absolutely continuous, then all $\FF$-semimartingales are $\GG$-semimartingales, and in particular $(V_t)_{t\geq 0}$ satisfies \eqref{MAPGOUSDE}. \\
A more thorough study of the semimartingale property of $(U,L)$ in the enlarged filtration $\GG$, or of absolute continuity of exponential functionals of MAPs, goes beyond the scope of this paper and will be left open for future research.
\end{remark}

\section{The Markov-modulated GOU process in risk theory}\label{S4}
\setcounter{equation}{0}

The classical (perturbed) \emph{Cram\'er-Lundberg risk process} describes the capital of an insurance company at time $t$ by
$$R_t=u+pt-\sum_{i=1}^{N_{L,t}} S_{L,i}+B_{L,t},\quad t\geq 0,$$
where $u>0$ is the initial capital, $p>0$ a constant premium rate, $(N_{L,t})_{t\geq 0}$ a Poisson process with rate $\lambda_L>0$ counting the number of claims up to a given time $t$, and $\{S_{L,i}, i\in\NN\}$ is an i.i.d. sequence of positive random variables, independent of $(N_{L,t})_{t\geq 0}$, that represent the claim sizes. The perturbation $(B_{L,t})_{t\geq 0}$ is a Brownian motion with variance $\sigma_L^2\geq 0$ and it is supposed to be independent of the claim process. Hence in this model
\begin{equation}\label{paulsen1}R_t=u+L_t, \quad t\geq 0,
\end{equation}
for the surplus generating spectrally negative L\'evy process $(L_t)_{t\geq 0}$ given by
\begin{equation} 
\label{paulsenL} L_t=pt- \sum_{i=1}^{N_{L,t}} S_{L,i}+B_{L,t}, \quad t\geq 0.\end{equation}
In \cite{PAULSEN_RisksStochasticEnvironment1992, GjessingPaulsen} Paulsen and Gjessing introduced and studied a combination of the Cram\'er-Lundberg model with an investment model. In this combined model it is assumed that the insurance company invests its capital in a market whose evolution can be described by the SDE
\begin{equation}\label{paulsen2}
dZ_t=Z_{t-}dU_t, \quad t\geq 0,
\end{equation}
where the investment returns are generated by another L\'evy process $(U_t)_{t\geq 0}$, that is assumed to be independent of the surplus generating process $(L_t)_{t\geq 0}$, and given by
\begin{equation} \label{paulsenU} U_t= rt+ \sum_{i=1}^{N_{U,t}} S_{U,i}+B_{U,t}, \quad t\geq 0. \end{equation}
Here $r\in\RR$ is a constant interest rate, $(\sum_{i=1}^{N_{U,t}} S_{U,i})_{t\geq 0}$ is a compound Poisson process with rate $\lambda_{U}\geq 0$, and $(B_{U,t})_{t\geq 0}$ is a Brownian motion with variance $\sigma_U^2\geq 0$, independent of $(\sum_{i=1}^{N_{U,t}}
S_{U,i})_{t\geq 0}$. Note that the SDE \eqref{paulsen2} is just the SDE \eqref{StochExpSDE} of the stochastic exponential and hence we have to assume $S_{U,1}>-1$ a.s. to prevent $Z_t$ from flipping its sign or dropping to zero due to a jump of $U$, which is a natural assumption.\\
Combining \eqref{paulsen1} and \eqref{paulsen2} yields  \emph{Paulsen's risk process} with stochastic investment described by the SDE \eqref{GOUSDELevy} of the GOU process with independent L\'evy processes $U$ and $L$. Since it is assumed that $U$ has no jumps less or equal to $-1$, this risk process is a generalized Ornstein-Uhlenbeck process \eqref{GOULevyexplizit} driven by another bivariate L\'evy process $(\xi,\eta)$.

In the following, we will apply the MMGOU process introduced in Section \ref{S1} above, to study a risk process with stochastic returns on the investment in a Markov-modulated environment, i.e. we consider a Markov-switching variation of  Paulsen's risk process described above.

 We start by defining a bivariate $\FF$-MAP $((U_t,L_t),J_t)_{t\geq 0}$ whose Markovian component $(J_t)_{t\geq 0}$ is ergodic and defined on a finite state space $S$. Whenever $J_t=j$ on some time interval $(t_1,t_2)$, the additive component $(U_t, L_t)_{t_1<t<t_2}$ is supposed to behave in law as a L\'evy process $(U^j,L^j)$ with independent components of the forms \eqref{paulsenU} and \eqref{paulsenL}, respectively. The parameters of these switching L\'evy processes are $p^j>0$, $\lambda_{L^j}\geq 0$, $\sigma_{L^j}^2\geq 0$, $r^j\in \RR$, $\lambda_{U^j}\geq 0$, and $\sigma_{U^j}^2\geq 0$, respectively, the jump sizes are denoted by $\{S_{L^j,i}, i\in \NN\}$, and $\{S_{U^j,i}, i\in \NN\}$. Moreover, whenever $J$ jumps, say at time $T_n$ from $i$ to $j$, this jump induces a simultaneous additional jump $(\Phi_{U,n}^{ij}, \Phi_{L,n}^{ij})$ in $(U,L)$. These jumps can be interpreted as market shocks associated to the regime switches. Again, to avoid an unnatural flipping of the sign or dropping to zero of the included investment process, we have to assume that $\Phi_{U,n}^{ij}>-1$ a.s. Note that due to the additional jumps induced by $J$, the components $U$ and $L$ of the additive component are in general not conditionally independent given $J$, see also Remark \ref{rem-independencepurejumps}.

The resulting risk process $(V_t)_{t\geq 0}$ is supposed to be governed by the SDE \eqref{MAPGOUSDE} of the MMGOU process. From the above construction it follows immediately (see also Remark \ref{rem-naive}), that as long as $J_t=j$ on some time interval $(t_1,t_2)$, this risk process behaves in law as Paulsen's risk process described above, driven by the surplus generating process $L^j$ and the return on investments generating process $U^j$.
Moreover, from Theorem \ref{PropMAPGOUSDE} it follows that the risk process  $(V_t)_{t\geq 0}$ can be represented by \eqref{MAPGOUexplicitxi} for another bivariate $\FF$-MAP $((\xi,\eta),J)$. In particular, as long as $J_t=j$ on some time interval $(t_1,t_2)$, the additive component $(\xi,\eta)$ behaves in law as the bivariate L\'evy process
\begin{align}
\begin{pmatrix} \xi_t^j \\ \eta^j_t \end{pmatrix}
&= \begin{pmatrix} -U_t^j  + \frac12 \sigma_{U^j}^2 t + \sum_{0< s\leq t}(\Delta U^j_s - \log (1+\Delta U^j_s)) \\ L_t^j \end{pmatrix} \label{paulsenxieta1}\\
  &= \begin{pmatrix} \big(\frac12 \sigma_{U^j}^2 -r^j\big) t  - \sum_{i=1}^{N_{U^j,t}}  \log (1+S_{U^j,i})-  B_{U^j,t} \\ p^jt- \sum_{i=1}^{N_{L^j,t}} S_{L^j,i}+B_{L^j,t}
 \end{pmatrix} ,\quad t\geq 0,\label{paulsenxieta2}
\end{align}
which follows by straightforward computations from $e^\xi=\cE(U)$ together with \eqref{exponentialexplicit}, \eqref{eq-etaviaUL}, and the special forms \eqref{paulsenU} and \eqref{paulsenL} chosen for the independent processes $U^j$ and $L^j$, respectively. The additional jumps of $(\xi,\eta)$ induced by jumps of $J$ are likewise computed as
\begin{equation}
(\Phi_{\xi,n}^{ij}, \Phi_{\eta,n}^{ij})\mathds{1}_{\lbrace J_{T_{n-1}}=i,J_{T_{n}}=j \rbrace} = (\Delta \xi_{T_n},\Delta \eta_{T_n}) = \left(-\log(1+\Phi_{U,n}^{ij}), \Phi_{L,n}^{ij}(1+\Phi_{U,n}^{ij})^{-1} \right)\mathds{1}_{\lbrace J_{T_{n-1}}=i,J_{T_{n}}=j \rbrace}, \label{paulsenxieta3}
\end{equation}
 for all $i,j\in S$, $n\in \NN$.

In the sequel we will denote the risk process $(V_t)_{t\geq 0}$ by $(V_t^u)_{t\geq 0}$, whenever the starting capital is chosen as $V_0=u>0$ a.s. The \emph{ruin probability} of the risk process with initial capital $u$ and initial regime $j\in S$ is given by
\begin{equation} \label{eq-ruinprob} \Psi_j(u):=\PP_j(V_t^u\leq 0 \; \text{for some }t\geq 0) = \PP_j(\tau_u<\infty), \end{equation}
for the \emph{time of ruin} 
$$\tau_u:= \inf\{t\geq 0: V_t^u\leq 0\},$$
which is a stopping time with respect to $\FF$, cf. \cite[Thm. I.4]{PROTTER_StochIntandSDE}. Moreover it is clear from the representation \eqref{MAPGOUexplicitxi}, that
\begin{equation} \label{timeofruin} \tau_u=\inf\Big\{t\geq 0: \int_{(0,t]} e^{\xi_{s-}}d\eta_s \leq -u\Big\} = \inf\big\{t\geq 0: \frace_{(-\xi,\eta)}(t)\leq u \big\}.\end{equation}

We start with a simple lemma in which we show that ruin in our model happens with positive probability as soon as we rule out some practically irrelevant cases.

\begin{lem}
 In the above decribed model 
 $$\PP_j(\tau_u<\infty)>0 \quad \text{for all }j\in S,$$
 as long as there exists some $j\in S$ such that $L^j$ is not a subordinator, i.e. such that $\sigma_{L^j}^2>0$ or $\lambda_{L^j}>0$.
\end{lem}

\begin{proof}
Assume first that $j\in S$ is such that $L^j$ is not a subordinator. Then clearly by \eqref{paulsenxieta1} also $\eta^j$ is not a subordinator. As $p^j>0$, by \cite[Thm. 4.1]{BLRR} this implies that $\int_{(0,\sigma]} e^{\xi_{s-}^j} d\eta_s^j$ for an exponentially distribution random time $\sigma$, independent of $(\xi_t^j,\eta_t^j)_{t\geq 0}$, can not be bounded from below. Because of $\int_{(0,t]} e^{\xi_{s-}}d\eta_s \overset{d}= \int_{(0,t]} e^{\xi_{s-}^j}d\eta_s^j$ under $\PP_j$ for all $t< \tauex_1(j)$ this implies 
$$\PP_j \left( \int_{(0,\tauex_1(j)]} e^{\xi_{s-}}d\eta_s \leq -u \right) >0 \quad \text{for all }u>0,$$
with $\tauex_1(j) \sim \text{Exp}(-q_{jj})$ as defined in Section \ref{S3a}. In particular, ruin happens with positive $\PP_j$-probability, even before the first exit time of $j$.\\
If $j$ is such that $L^j$ is a subordinator, then it follows from \cite[Thm. 4.1]{BLRR} that $\int_{(0,\tauex_1(j)]} e^{\xi_{s-}^j} d\eta_s^j$ is indeed bounded from below by $0$. However, 
by assumption there exists another state, say $k\in S$, such that $L^k$ is not a subordinator. As $J$ is assumed to be ergodic, it will reach this state $k\in S$ in finite time. Since 
$$\int_{(0, \tauex_1(k)]}e^{\xi_{s-}} d\eta_s = \int_{(0,\taure_1(k)]} e^{\xi_{s-}} d\eta_s + e^{\xi_{\taure_1(k)}}\int_{(\taure_1(k), \tauex_1(k)]} e^{\xi_{s-}- \xi_{\taure_1(k)}} d\eta_s,$$
with the second integral being unbounded from below by the same reasoning as above, ruin will then happen with positive $\PP_j$-probability before exiting the state $k$.
\end{proof}

We end this paper with a formula for the ruin probability \eqref{eq-ruinprob} that generalizes \cite[Thm. 3.2]{PAULSEN_RisksStochasticEnvironment1992} to the Markov-modulated setting.

\begin{theorem}\label{thm-ruinMMpaulsen}
Assume that $\frace_{(-\xi,\eta)}(t)$ converges $\PP_\pi$-a.s. as $t\to \infty$ and denote the cdf of its limit given $J_0=j$ by 
$$H_j(x):=\PP_j\big(\frace_{(-\xi,\eta)}^\infty \leq x\big), \quad x\in \RR.$$	Then the ruin probability \eqref{eq-ruinprob} in the above described model is given by	
	$$\Psi_j(u) = \frac{H_j(-u)}{  \EE_j\left[ H_{J_{\tau_u}} (-V_{\tau_u}^u) \big| \tau_u<\infty \right] } \quad \text{for all }u>0, \;j\in S.$$
\end{theorem}

Before we present the proof of this theorem, let us discuss the assumption of almost sure convergence of $\frace_{(-\xi,\eta)}(t)$ in the context of our model:

\begin{remark} General  necessary and sufficient conditions for a.s. convergence of $\frace_{(-\xi,\eta)}(t)$ are given in \cite[Thm. 4.1]{BEHME+SIDERIS_ExpFuncMAP2020}. As the L\'evy processes $U^j$ and $L^j$ are assumed to be independent, by \cite[Prop. 4.7]{BEHME+SIDERIS_ExpFuncMAP2020} the exponential functional considered here can not be degenerate. This implies by \cite[Rem. 4.2]{BEHME+SIDERIS_ExpFuncMAP2020} that convergence in $\PP_j$-probability and a.s. convergence of  $\frace_{(-\xi,\eta)}(t)$ are equivalent. In particular, a necessary and sufficient condition for almost sure convergence is thus given by
$$ \lim_{t\in \TT_j, t\to \infty} \xi_t=-\infty \; \PP_j\text{-a.s. and } I^j_{(-\xi,\eta)}<\infty \quad \text{for some }j\in S,$$
with $I^j_{(-\xi,\eta)}$ as in \eqref{eq-nesssuffweak}. Still, as this condition is typically hard to check, note that by \cite[Prop. 5.2 and Prop. 5.7]{BEHME+SIDERIS_ExpFuncMAP2020}  a.s. convergence of $\frace_{(-\xi,\eta)}(t)$ also follows if the following slightly simpler conditions are satisfied:
\begin{enumerate}
	\item The long-term mean $\kappa_\xi$ of $\xi$ is well-defined and negative, i.e.
	\begin{align*}\kappa_\xi&= \sum_{j\in S} \pi_j \EE[\xi^j_1] + \sum_{\substack{(i,j)\in S\times S\\ i\neq j}} \pi_i q_{ij} \EE[\Phi_{\xi,1}^{ij}] \\
	&=\sum_{j\in S} \pi_j \left( \frac12 \sigma_{U^j}^2 - r^j - \lambda_{U^j} \EE[\log(1+S_{U^j,1})] \right) - \sum_{\substack{(i,j)\in S\times S\\ i\neq j}} \pi_i q_{ij} \EE\left[ \log(1+\Phi_{U,1}^{ij})\right] <0\end{align*}
	exists, where the equality follows from \eqref{paulsenxieta2} and \eqref{paulsenxieta3}.
	\item The following integral condition holds for some $j\in S$:
	\begin{align*}
	\int_{(1,\infty)} \frac{\log q}{\bar{A}_{-\xi} (\log q)} \PP_j\Bigg(\Bigg| \int_{(0,\taure_1(j)]} e^{\xi_{t-}} d\eta^d_t\, \Bigg|\in dq \Bigg)	<\infty, \end{align*}
	with $(\eta_t^d)_{t\geq 0}$ denoting  the resulting claim process including market shocks, i.e.
	$$\eta_t^d = \sum_{0< s\leq t} \Delta L_s^{J_s} + \sum_{n\geq 1}\sum_{\substack{(i,j)\in S\times S\\ i\neq j}} \frac{\Phi_{L,n}^{ij}}{1+\Phi_{U,n}^{ij}} \mathds{1}_{\lbrace J_{T_{n-1}}=i,J_{T_{n}}=j,T_n\leq t\rbrace}  , \quad t\geq 0,$$ 
	$e^{\xi}= \cE(U)^{-1}$, and  $\bar{A}_{-\xi}(x)$ in the present setting and due to \eqref{paulsenxieta2} and \eqref{paulsenxieta3} given by
	\begin{align*}\bar{A}_{-\xi}(x)&= \sum_{j\in\cS} \pi_j \Big( \gamma_{\xi^{j}} + \nu_{\xi^{j}}((1,\infty)) + \int_1^x \nu_{\xi^{j}}((y,\infty)) dy + \sum_{\substack{i\in\cS \\i\neq j}} q_{ij} \EE\left[(\Phi_{\xi,1}^{ij})^+\right]  \Big) \\
	&= \sum_{j\in S} \pi_j \Bigg( r^j-\frac12 \sigma_{U^j}^2 - \sum_{\substack{i\in S\\ i\neq j}} q_{ij} \EE\left[ \log^-(1+\Phi_{U,1}^{i,j})\right] \\ & \qquad + \lambda_{U^j} \Big(\int_{[-1+e^{-1},-1+e]} \log(1+y) F_{U^j}(dy) + \bar{F}_{U^j}(-1+e) + \int_1^x \bar{F}_{U^j} (-1+e^y) dy \Big) \Bigg),\end{align*}
	for $F_{U^j}(x)=\PP(S_{U^j,1}\leq x)$ denoting the cdf of the jumps of $U^j$, and $\bar{F}_{U^j}(x)= 1- F_{U^j}(x)$ the corresponding tail function.
\end{enumerate}
\end{remark}

\begin{proof}[Proof of Theorem \ref{thm-ruinMMpaulsen}]
	Obviously for any $t\geq 0$ it holds
	$$\frace_{(-\xi,\eta)}(t) = \int_{(0,t]} e^{\xi_{s-}} d\eta_s = \int_{(0,\infty)} e^{\xi_{s-}} d\eta_s - \int_{(t,\infty)} e^{\xi_{s-}} d\eta_s \quad \text{a.s.},$$
	and hence 
	$$ e^{-\xi_t} \left(u+ \int_{(0,t]} e^{\xi_{s-}} d\eta_s\right) +  \int_{(t,\infty)} e^{\xi_{s-}-\xi_t} d\eta_s  = e^{-\xi_t}\left(u+  \int_{(0,\infty)} e^{\xi_{s-}} d\eta_s \right) \quad \text{a.s.},$$
	i.e.
	$$ V_t^u + \int_{(t,\infty)} e^{\xi_{s-}-\xi_t} d\eta_s  = e^{-\xi_t}\left(u+  \frace_{(-\xi,\eta)}^\infty \right) \quad \text{a.s.}$$
	Moreover, as the event $\{\frace_{(-\xi,\eta)}^\infty \leq -u\}$ implies $\{\frace_{(-\xi,\eta)}(t) \leq -u \;\text{for some }t\geq 0\} = \{\tau_u<\infty\}$ we may therefore compute
	\begin{align}
	H_j(-u) &= \PP_j\left(\frace_{(-\xi,\eta)}^\infty \leq -u\right) =  \PP_j\left(\frace_{(-\xi,\eta)}^\infty \leq -u, \tau_u<\infty\right)  \nonumber \\
	&= \PP_j\left(  \int_{(\tau_u,\infty)} e^{\xi_{s-}-\xi_{\tau_u}} d\eta_s  \leq - V_{\tau_u}^u , \tau_u<\infty\right). \label{eq-paulsenproof1}
	\end{align}
	Now note that, given $J_{\tau_u}$, the exponential integral $\int_{(\tau_u,\infty)} e^{\xi_{s-}-\xi_{\tau_u}} d\eta_s$ is conditionally independent of $\tau_u$. Thus  by the law of total probability, and using the strong Markov property of the MAP $((\xi,\eta),J)$ (i.e. \eqref{DefMAP} or equivalently \eqref{MAPpropextend} for $s$ replaced by a stopping time) as it has been proven in \cite[Thm. 3.2]{CINLAR_MAP21972}, we obtain
	\begin{align*}
	\PP_j\left(  \int_{(\tau_u,\infty)} e^{\xi_{s-}-\xi_{\tau_u}} d\eta_s  \in \cdot \,, \tau_u<\infty\right) &= \sum_{k\in S} \PP_j(J_{\tau_u}=k) \PP_j\left(  \int_{(\tau_u,\infty)} e^{\xi_{s-}-\xi_{\tau_u}} d\eta_s  \in \cdot \,, \tau_u<\infty \Bigg|J_{\tau_u}=k\right)\\
	&= \sum_{k\in S} \PP_j(J_{\tau_u}=k) \PP_j\left(  \int_{(0,\infty)} e^{\xi'_{s-}} d\eta'_s  \in \cdot \,, \tau_u<\infty \Bigg|J_{\tau_u}=k=J'_0\right),
	\end{align*}
	for some independent copy $((\xi',\eta'),J')$ of $((\xi,\eta),J)$. Inserting this in \eqref{eq-paulsenproof1}  yields
	\begin{align*}
	H_j(-u)&= \sum_{k\in S} \PP_j(J_{\tau_u}=k) \PP_j\left(\frace_{(-\xi',\eta')}^\infty \leq -V_{\tau_u}^u , \tau_u<\infty \Big| J_{\tau_u}=k =J'_0\right)\\
	&= \sum_{k\in S} \PP_j(J_{\tau_u}=k) \PP_j\left(\frace_{(-\xi',\eta')}^\infty \leq -V_{\tau_u}^u \Big| \tau_u<\infty , J_{\tau_u}=k=J'_0 \right) \PP_j(\tau_u<\infty)\\
	&=  \sum_{k\in S} \PP_j(J_{\tau_u}=k)  \EE_j\left[ H_k (-V_{\tau_u}^u) \Big| \tau_u<\infty, J_{\tau_u}=k \right] \PP_j(\tau_u<\infty)\\
	&=    \EE_j\left[ H_{J_{\tau_u}} (-V_{\tau_u}^u) \Big| \tau_u<\infty \right] \PP_j(\tau_u<\infty),
	\end{align*}
	and solving this equation for the ruin probability finishes the proof.
\end{proof}

	\section*{Acknowledgements}
	
	The authors thank Paolo Di Tella for comments on an earlier draft of this paper that lead to improvements of the manuscript. 
	


	\end{document}